\documentclass{amsart}
\usepackage[hmargin=3.5cm,vmargin=3.5cm]{geometry}

\usepackage[utf8]{inputenc}

\usepackage{amsmath,amssymb,amsthm, mathtools,dsfont}

\usepackage{graphicx}

\usepackage{enumerate,blindtext}
\usepackage[linktocpage=true]{hyperref}

\usepackage{todonotes, lipsum}
\usepackage[all,cmtip]{xy}
\usepackage{MnSymbol}

\makeatletter
\g@addto@macro\bfseries{\boldmath}
\makeatother

\hypersetup{
bookmarksnumbered=true
}

\newtheorem{theorem}{Theorem}
\newtheorem{lemma}{Lemma}[section]
\newtheorem{corollary}[lemma]{Corollary}

\newtheorem{proposition}[lemma]{Proposition}
\theoremstyle{definition}

\newtheorem{remark}[lemma]{Remark}
\newtheorem{definition}[lemma]{Definition}

\newtheorem*{korollar*}{Corollary}

\newcommand{\NN}{\mathbb{N}}

\newcommand{\RR}{\mathbb{R}}
\newcommand{\CC}{\mathbb{C}}

\newcommand{\cS}{\mathcal S}
\newcommand{\cE}{\mathcal E}
\newcommand{\cO}{\mathcal O}

\newcommand{\cF}{\mathcal F}

\newcommand{\cD}{\mathcal D}
\newcommand{\cH}{\mathcal H}
\newcommand{\cC}{\mathcal C}

\newcommand{\fg}{\mathfrak g}
\newcommand{\fh}{\mathfrak h}
\newcommand{\fr}{\mathfrak r}

\newcommand{\fm}{\mathfrak m}
\newcommand{\fn}{\mathfrak n}

\newcommand{\fz}{\mathfrak z}
\newcommand{\fw}{\mathfrak w}

\DeclareMathOperator{\im}{Im}

\DeclareMathOperator{\Id}{Id}

\DeclareMathOperator{\re}{Re}
\DeclareMathOperator{\wf}{WF}

\DeclareMathOperator{\tr}{Tr}

\DeclareMathOperator{\supp}{supp}

\DeclareMathOperator{\ac}{AC}
\DeclareMathOperator{\Ad}{Ad}
\DeclareMathOperator{\ad}{ad}

\DeclareMathOperator{\Ind}{Ind}
\newcommand{\matz}[4]{\left(\begin{array}{cc}
                             #1&#2\\
                             #3&#4
                            \end{array}\right)}
\newcommand{\matd}[9]{\left(\begin{array}{ccc}
                             #1&#2&#3\\
                             #4&#5&#6\\
                             #7&#8&#9
                            \end{array}\right)}
\newcommand{\Abb}[4]{\begin{cases}\begin{aligned} #1 & \rightarrow  #2 \\ #3 &\mapsto  #4\end{aligned}\end{cases}}

\title{Wave Front Sets of Nilpotent Lie Group Representations}

\author{Julia Budde and Tobias Weich}
\email{jbudde@math.uni-paderborn.de, weich@math.uni-paderborn.de}
\begin{document}
\begin{abstract}
Let $G$ be a nilpotent, connected, simply connected Lie group with Lie algebra $\mathfrak g$, and $\pi$ a unitary representation of $G$. In this article we prove that the wave front set of $\pi$ coincides with the asymptotic cone of the orbital support of $\pi$, i.e. $\mathrm{WF}(\pi)=\mathrm{AC}(\bigcup_{\sigma\in \mathrm{supp}(\pi)}\mathcal O_\sigma)$, where $\mathcal O_\sigma\subset i\mathfrak g^\ast$ is the coadjoint Kirillov orbit associated to the irreducible unitary representation $\sigma\in \hat{G}$.
\end{abstract}

\maketitle

\tableofcontents

\section{Introduction}
The concept of wave front sets was introduced by Sato and Hörmander. Given a distribution $u\in \mathcal D'(M)$ its wave front set is a closed conical subset $\wf(u)\subset T^*M$ that encodes the singularities of the distributions $u$. Informally speaking one can consider the wave front set as those directions in which the distribution is not smooth (in a $C^\infty$ sense). Wave front sets are extensively used in PDE theory as a very concise measure of singularities. For example Hörmanders famous theorem about propagation of singularities is formulated in terms of wave front sets.

The concept of the wave front set for a unitary Lie group representation was introduced by Howe \cite{howe}\footnote{For compact Lie groups a very similar concept based on the analytic instead of the $C^\infty$ regularities was introduced slightly before by Kashiwara and Vergne in \cite{KV79}}. Given a Lie group $G$ with Lie algebra $\fg$ and a unitary representation $(\pi,\mathcal H)$ the wave front set of the representation yields a closed $\Ad^*(G)$-invariant cone $\wf(\pi)\subset i\fg^*$. Informally speaking it captures the singular directions of all matrix coefficients of $\pi$ (see Definition~\ref{def:WFpi} for a precise definition). The remarkable property of $\wf(\pi)$ is that it is defined entirely in terms of singularities of matrix coefficients but it captures essential information of the spectral measure of $\pi$. This relation can be expressed by certain \emph{wave front-orbital support (WFOS) theorems} which we want to explain next: Suppose that the Lie group $G$ is of type I such that we can write any unitary representation $(\pi, \mathcal H)$ as a direct integral $\pi= \int^\oplus_{\hat G} \sigma^{m(\sigma)}\mu_\pi(\sigma)$ where $\hat G$ is the unitary dual endowed with the Fell topology and $\mu_\pi$ a Borel measure on $\hat G$, the spectral measure of $\pi$. Suppose furthermore that there is a canonical way to associate to any $\sigma\in \supp\mu_\pi\subset \hat G$ a coadjoint orbit $\mathcal O_\sigma\subset i\fg^*$ (or possibly a finite collection of such orbits), then we define the \emph{orbital support} to be
\begin{equation}\label{def:orbital_support}
\cO-\supp(\pi) := \bigcup_{\sigma\in \supp(\mu_\pi)}\cO_\sigma\subset i\fg^*.
\end{equation}
Furthermore, we define for any subset $S\subset i\fg^*$ its asymptotic cone
\[
 \ac(S) :=\{\xi\in i\fg^*|\mathcal C \text{ an open cone containing }\xi \Rightarrow S\cap\mathcal C\text{ unbounded }\}\cup\{0\}.
\]
A Wave front-orbital support theorem is then a theorem that states (for a suitable class of Lie groups $G$ and unitary representations $(\pi, \mathcal H)$) the equality
\begin{equation}
 \label{eq:wf_plancherel}
 \wf(\pi) = \ac(\cO-\supp(\pi))
\end{equation}

and thus connects the wave front set to the asymptotic support of the spectral measure. For abelian Lie groups the WFOS-theorem is just a reflection of the definition of the wave front set and Fourier inversion formulas as had been noted by Howe \cite{howe}. For non-commutative Lie groups the relation is much more subtle and has been shown for compact groups by Kashiwara-Vergne \cite{KV79}\footnote{with their slightly different notion of wavefront set, as mentioned above} and Howe~\cite{howe}. Much more recently Harris, He and \'Olafsson \cite[Theorem 1.2]{harrisheolaf} have shown a WFOS-theorem for real reductive algebraic groups $G$ and unitary representations $\pi$ which are weakly contained in the tempered representations (see \cite{BenII, BenYoshiki, BenYoshiki22} for follow up works that aim to weaken the temperedness assumption).

The practical purpose of WFOS-theorems is that they connect the spectral measure $\mu_\pi$ of general unitary representations to the wave front set of $\pi$. While the former is in general very difficult to determine, the latter has been shown to be explicitly calculable in very general settings. For example if $G$ is an arbitrary Lie group  and $H\subset G$ a closed subgroup such that $G/H$ carries a non-vanishing $G$-invariant smooth density then one can consider the regular representation of $G$ on $L^2(G/H)$. While determining the exact spectral measure (i.e. the Plancherel measure) of $L^2(G/H)$ is in general extremely difficult and so far only known for certain classes of homogeneous spaces, the wave front set of $L^2(G/H)$ is known \cite[Theorem 2.1]{harrisweich}  without any further assumptions 

\[
 \wf(L^2(G/H)) = \overline{ \Ad^*(G)i(\fg/\fh)^*}.
\]
Similar identities have also been derived for certain classes of induced representations \cite[Theorem 2.2 and 2.3]{harrisweich} and also the behavior of wave front sets under restrictions is rather well understood \cite[Prop 1.5]{howe}\cite[Corollary 1.4]{harrisheolaf}. Combining the explicit knowledge of $\wf(L^2(G/H))$ with a WFOS-theorem one can then deduce results about the Plancherel measures, e.g. existence of discrete series (see e.g. \cite[Example 7.5]{harrisweich}\cite[Theorem 21.1]{DKKS18}\cite{BenYoshiki22}).

In contrast to the knowledge about $\wf(L^2(G/H))$ that is known without any structural assumptions on $G$ and only mild assumptions on the quotient $G/H$, the cases in which WFOS-theorems are established are rather limited (abelian \cite{howe}, compact \cite{howe, KV79} and real reductive groups \cite{harrisheolaf} as mentioned above). One might hope that they can be proven for any class of Lie groups where a suitable relation between unitary irreducible representations and coadjoint orbits is established, for example in the setting of real linear algebraic group (see e.g. \cite{Duf10}). The purpose of this article is to establish a WFOS-theorem for nilpotent Lie groups. We prove

\begin{theorem}\label{thm:wf=ac}
 Let $G$ be a nilpotent, connected, simply connected Lie group and $\pi$ a unitary representation of $G$.  Then
 \begin{eqnarray*}
 	\wf(\pi) = \ac(\cO-\supp\pi).
 \end{eqnarray*}
Where the orbit support	\eqref{def:orbital_support} is defined by the Kirillov orbits $\mathcal O_\sigma\subset i\fg^*$ of the unitary irreducible representation $\sigma$.
\end{theorem}
It was rather surprising to us, that the proof strategy of \cite{harrisheolaf} could not be transferred to the setting of nilpotent Lie groups. A central object in the proof of the Wave front-Plancherel theorem in \cite{harrisheolaf} was the analysis of integrated characters\footnote{Such integrated characters had before been introduced and used in the context of restriction problems by Kobayashi \cite{ToshiI, ToshiII, ToshiIII}.} $\int_{\hat G} \chi_\sigma f(\sigma)d\mu_\pi(\sigma)\in \mathcal D'(G)$ where $\chi_\sigma\in \mathcal D'(G)$ is the distributional character of the tempered irreducible representation $\sigma$. Harris, He and \'Olaffson then use character formulas of Duflo and Rossmann as well as Harish-Chandra's invariant integrals to relate the wave front set of the integrated characters to the asymptotic orbital support. While Kirillov's character formula provides a natural (and even simpler) replacement to the Duflo-Rossmann formula, the analog to the Harish-Chandra invariant integrals for nilpotent groups produces additional singularities which make the proof break down (see \cite[Section 5.1]{budde21} for a detailed discussion of the occurring problems). We therefore had to establish an alternative method to prove the above result. Instead of working with integrated characters and character formulas we directly work with matrix coefficients. In contrast to the characters, the Fourier transform of individual matrix coefficients of an irreducible representation are not supported on the coadjoint orbits. However we can show (Proposition~\ref{prop:ReMcoeffsgeqSimpl} and Proposition~\ref{prop:detaO}) that they are microlocally supported ``near'' the orbit and that the precise meaning of ``near'' can be made uniform about all unitary representations. Our proof of these key propositions is based on concrete microlocal estimates on induced representations. The induction scheme hereby is similar to the induction in the traditional proof of Kirillov's character formula.

\emph{Acknowledgments}
We thank Benjamin Harris, Joachim Hilgert, Jan Frahm and Clemens Weiske for many encouraging discussions and helpful remarks and suggestions. 
This project has received funding from Deutsche Forschungsgemeinschaft (DFG) (Grant No. 422642921  Emmy Noether group ``Microlocal Methods for Hyperbolic Dynamics'') as well as from SFB-TRR 358/1 2023 — (Grant No. 491392403) (CRC ``Integral Structures in Geometry and Representation Theory'').

\section{Preliminaries}
\subsection{Wave Front Sets}\label{sec:wf_prelim}
In this section we give definitions of the wave front set of a distribution and of a unitary Lie group representation and provide some facts about these objects that we will use later in the article.

Let $W$ be a real, finite-dimensional vector space and fix a Lebesgue measure $dx$ on $W$.
We define the Fourier transform as the map   $\cF: \cS(W)\to\cS(iW^\ast)$ between Schwartz spaces with
\begin{align*}
	\cF(\varphi)(\zeta):=\int_W \varphi(x)e^{-2\pi\langle \xi,x\rangle}~dx, \quad \xi\in iW^\ast,
\end{align*}
and for a tempered distribution $u\in\cS'(W)$ as $\cF(u)\in\cS'(iW^\ast)$ with
$\cF(u)(\psi):=u(\cF(\psi))$ for  $\psi\in\cS(iW^\ast)$. 
The inversion formula for $\cF:\cS(W)\to\cS(iW^\ast)$ gives us
\begin{align*}
	\cF^{-1}:\cS(iW^\ast)\to\cS(W), \quad \psi\mapsto \left(x\mapsto\int_{iW^\ast} \psi(\xi)e^{2\pi\langle\xi,x\rangle}~d\xi\right)
\end{align*}
for a suitable measure $d\xi$ on $iW^\ast$.

In addition to that, we define the Fourier transform of a distribution $v\in\cE'(W)$ with compact support to be	
\begin{align*}
	\cF(v)(\xi):= v\left\lbrack e^{-2\pi \langle \xi,\bullet\rangle}\right\rbrack, \quad \xi\in iW^\ast.
\end{align*}

\begin{definition} \label{def:WFvs}
Let $W$ be a real, finite-dimensional vector space and $u\in \cD'(X)$ a distribution on an open subset $X\subset W$. Then we say $(x_0,\xi_0)\in X\times iW^\ast\setminus \{0\} \subset iT^\ast X$ is \emph{not} in the \emph{wave front set} $\operatorname{WF}(u)\subset iT^\ast X$ if there exist open neighborhoods $U$ of $x_0$ and $V$ of $\xi_0$ and a smooth compactly supported function $\phi\in C_c^{\infty}(U)$ with $\phi(x_0)\neq 0$ such that for all $N\in\mathbb{N}$ there exists a constant $C_{N,\phi}>0$ such that
$$|\mathcal{F}(\phi u)(\tau \xi)|\leq C_{N,\phi} |\tau|^{-N}\quad \forall\, \tau\gg 0, \xi\in V.$$
\end{definition}
Note that  $(x,0)\in iT^\ast X$ is never in the wave front set (contrary to Definition~\ref{def:WFpi} for unitary representations) because in order to analyze the singularities of a function or distribution it only makes sense to look in the directions $\xi\neq 0$. \\
Furthermore, it is easily seen from the definition that the wave front set $\operatorname{WF}(u)\subset iT^\ast X$ is a closed cone (in the second component).

Now, if $\psi:X\to Y$ is a diffeomorphism between two open sets and $u$ is a distribution on $Y$, then	$\psi^\ast \wf(u)=\wf(\psi^\ast u)$,
where the pullback on the cotangent bundle is defined by
\begin{align*}
	\psi^\ast(y,\xi)=\left( \psi^{-1}(y),(D\psi(\psi^{-1}(y)))^T\xi \right), \quad (y,\xi)\in iT^\ast Y.
\end{align*}
Thus, the notion of the wave front set of a distribution on a smooth manifold is independent of the choice of local coordinates and is therefore well-defined.

Now let $G$ be a $n$-dimensional Lie group with Lie algebra $\fg$ and $(\pi,\cH)$ a unitary representation of $G$.
Denote by  $J_1(\cH)$ the space of trace class operators with trace class norm $\|T\|_1$.

\begin{definition}\label{def:WFpi}
The \emph{wave front set of a unitary representation $\pi$} is defined as the closure of the union of the wave front sets at the identity of the matrix coefficients of $\pi$: 
\begin{align*}
	\wf(\pi) :=  \overline{\bigcup_{v,w\in\cH}\wf_e(\langle \pi(g)v,w\rangle_\cH)} \cup \{0\} \subset iT^\ast_e G \cong i\fg^\ast.
\end{align*}
Here we use the convention that zero is always in the wave front set (contrary to Definition~\ref{def:WFvs}) because it makes the statements of the results for unitary representations cleaner. \\
Howe used in \cite{howe} the equivalent definition 
\begin{align*}
	\wf(\pi) = \overline{\bigcup_{T\in J_1(\cH)}\wf_e(\tr_\pi(T))}\cup\{0\} ,
\end{align*}
where  $\tr_\pi(T):=\tr(\pi(\cdot)T)$, $T\in J_1(\cH)$, is a continuous bounded function on $G$ regarded as a distribution on $G$ by integration. The equivalence of these definitions was shown in \cite[Proposition 2.4]{harrisheolaf}.
\end{definition}

It is a well-known fact that the wave front set $\wf(\pi)\subset i\fg^\ast$ is a closed, $\Ad^\ast(G)$-invariant cone.

The following result provides another description of the wave front set which we will use in our proof.

\begin{lemma}[see {\cite[Theorem 1.4 v)]{howe}} and {\cite[Lemma 2.5 (iii)]{harrisheolaf}}]\label{lem:HoweKrits}~\\
Let $\pi\neq 0$ and $\xi_0 \in i\fg^\ast$. Then $\xi_0\notin \wf(\pi)$ if and only if there are an open set $e\in U\subset G$ on which the logarithm is a well-defined diffeomorphism onto its image and an open set $\xi_0\in V\subset i\fg^\ast$ such that for every $\phi\in C_c^\infty(U)$ there exists a family of constants $C_N(\phi)>0$ independent of both $\xi\in V$ and $T\in J_1(H)$, such that
\begin{align*}
		\left|\int_G \tr_\pi(T)(g)e^{-2\pi\tau\xi(\log g)}\phi(g)~dg\right| \leq C_N(\phi)\|T\|_1 \tau^{-N}
\end{align*}
for $\tau\gg 0$, $\xi\in V$, $T\in J_1(H)$.
\end{lemma}
For our proof in Section \ref{subsec:1stIncl} we need to know more about the dependence of the constant $C_N(\phi)$ on the cut-off function $\phi\in C_c^\infty(G)$.
\begin{lemma}\label{lem:HoweKritsPhi}
For all $N>n=\dim(G)$ the above statement holds with the choice of the constant $C_N(\phi)=C_N \|\phi\|_{W^{N+n,1}}$ where  $\|\phi\|_{W^{M,1}}:=\sum_{|\alpha|\leq M}\|D^\alpha\phi\|_{L^1}$ is a Sobolev norm.
\end{lemma}
\begin{proof}
We may assume without loss of generality that in Lemma \ref{lem:HoweKrits} $V=B_{2\epsilon}(\xi_0)$ for an $\frac 12 >\varepsilon>0$ and $\|\xi_0\|=1$, and may prove our statement for  $\xi\in V'\coloneq  B_{\epsilon}(\xi_0)$ with $\|\xi\|=1$.
Now, let $U\subset G$ be the open set given by Lemma \ref{lem:HoweKrits} and take $U'\subsetneqq U$ open and $\chi\in C_c^\infty(\log(U))$ a function on $\fg$ with $\chi=1$ on $\log(U')\subset\fg$. Then we can estimate for all $\phi\in C_c^\infty(U')$, $\varphi=\phi\circ\exp\in C_c^\infty(\fg)$:
\begin{align*}
I(\phi,\xi,T)(\tau)&:= \int_G \tr_\pi(T)(g)e^{-2\pi\tau\xi(\log g)}\phi(g)~dg 
=\int_\fg \tr_\pi(T)(\exp(X))e^{-2\pi\tau\xi(X)}\chi(X)\varphi(X)~dX \\
&= \int_{i\fg^\ast} \left(\int_\fg \tr_\pi(T)(\exp(X))\chi(X)e^{-2\pi\eta(X)} dX\right) \left(\int_\fg \varphi(Y)e^{2\pi(\eta-\tau\xi)(Y)} dY\right) d\eta,
\end{align*}
and define $J_1(\eta)\coloneq \int_\fg \tr_\pi(T)(\exp(X))\chi(X)e^{-2\pi\eta(X)} dX$ and $J_2(\eta)\coloneq \int_\fg \varphi(Y)e^{2\pi(\eta-\tau\xi)(Y)} dY$.
With the $\frac 12 >\varepsilon>0$ chosen above we split up the integral as $I(\phi,\xi,T)(\tau)=I_1+I_2$ where
\begin{align*}
I_1\coloneq \int_{\|\tau\xi-\eta\| \geq \varepsilon\tau} J_1(\eta)J_2(\eta) d\eta, \quad 
I_2\coloneq \int_{B_{\varepsilon\tau}(\tau\xi)} J_1(\eta)J_2(\eta) d\eta.
\end{align*}
For the first integral we estimate for $\eta\notin B_{\varepsilon\tau}(\tau\xi)$ by estimation of the integrand and partial integration, respectively
\begin{align*}
|J_1(\eta)|\leq \|T\|_1\|\chi\|_{L^1}, \quad 
|J_2(\eta)|\leq \|\varphi\|_{W^{N,1}} \|\tau\xi-\eta\|^{-N}
\end{align*}
and therefore
\begin{align*}
|I_1| &\leq \|T\|_1\|\chi\|_{L^1} \|\varphi\|_{W^{N,1}} \int_{\|\tau\xi-\eta\| \geq \varepsilon\tau} \|\tau\xi-\eta\|^{-N} d\eta
= \|T\|_1\|\chi\|_{L^1} \|\varphi\|_{W^{N,1}} \mathrm{vol}_{n-1}(S^{n-1})\int_{\varepsilon\tau}^\infty r^{-N}r^{n-1} dr \\
&=\frac{1}{N-n} \|T\|_1\|\chi\|_{L^1} \|\varphi\|_{W^{N,1}} \mathrm{vol}_{n-1}(S^{n-1})\varepsilon^{-N+n} \tau^{-N+n}.
\end{align*} 
For the second integral we estimate for $\eta\in  B_{\varepsilon\tau}(\tau\xi)$ with Lemma \ref{lem:HoweKrits} applied to $\chi$ and $\frac 1{\|\eta\|}\eta\in B_{2\epsilon}(\xi_0)=V$
\begin{align*}
|J_1(\eta)|\leq \|T\|_1 C_N(\chi)\|\eta\|^{-N}, \quad 
|J_2(\eta)|\leq \|\varphi\|_{L^1} \leq \|\varphi\|_{W^{N,1}}.
\end{align*}
Since $\|\eta\|\geq (1-2\varepsilon)\tau$ we have
\begin{align*}
|I_2| &\leq \|T\|_1 C_N(\chi) \|\varphi\|_{W^{N,1}} \int_{B_{\varepsilon\tau}(\tau\xi)} \|\eta\|^{-N}d\eta \\
&\leq \|T\|_1 C_N(\chi) \|\varphi\|_{W^{N,1}} ((1-2\varepsilon)\tau)^{-N} C (\varepsilon\tau)^n
\leq C_N \|T\|_1 \|\varphi\|_{W^{N,1}} \tau^{-N+n}.
\end{align*}
This proves the statement with $U'$ as the open neighborhood of $e\in G$ and $V'$ as the open neighborhood of $\xi_0$ in $i\fg^\ast$.
\end{proof}

Lastly, the following simple result gives us an idea why wave front sets might be interesting for the decomposition of unitary representations.

\begin{proposition}\label{prop:wfoplus}
Let $(\pi_1,\cH_1)$,\ldots, $(\pi_k,\cH_k)$ be unitary representations of $G$, then
\begin{eqnarray*}
\wf\left(\bigoplus_{j=1}^k\pi_j\right) = \bigcup_{j=1}^k \wf(\pi_j).
\end{eqnarray*}
\end{proposition}

\subsection{Nilpotent Lie Groups}\label{sec:nilpotent_prelim}
In order to prove Theorem \ref{thm:wf=ac} we use the structure theory of nilpotent Lie algebras and Lie groups. The required results below are mostly from  the book by Corwin and Greenleaf \cite{corgre}.

Let $G$ be a nilpotent, connected, simply connected Lie group with Lie algebra $\fg$ of dimension $n$ and $\fg^\ast$ its vector space dual. By $\hat{G}$ we denote the unitary dual of $G$ and by $i\fg^\ast/G$ the space of coadjoint orbits.

The main results are the following two theorems:

\begin{theorem}[see {\cite[Theorems 2.2.1 - 2.2.4]{corgre}}]\label{thm:BIJ}
	There exists a homeomorphism $\hat{G} \to i\fg^\ast/G$, $\sigma\mapsto\cO_\sigma$, and $\sigma_l \leftmapsto \cO_l=\Ad^\ast(G)l$.
For the continuity of the map $i\fg^\ast/G\to\hat{G}$ see \cite[Theorem 8.2]{kirillov} and for the continuity of the map $\hat{G}\to i\fg^\ast/G$ see \cite{brown}.
\end{theorem}

The structure and parametrization of the coadjoint orbits is given by

\begin{theorem}[see {\cite[Theorem 3.1.14]{corgre}}]\label{thm:AbkParamO}
Fix a (strong Malcev) basis $\{X_1,\ldots ,X_n\}$ of $\fg$. Then there exists a finite set $D$ of orbit types. Denote by $U_d\subset i\fg^\ast$ the union of all orbits of type $d\in D$. Moreover, all orbits in $U_d$ have the same dimension $d_n$. \\
For each $d\in D$ there also exists a cross-section $\Sigma_d\subset i\fg^\ast$ of the orbits in $U_d$, i.e. each orbit $\cO\subset U_d$ intersects $\Sigma_d$ in a unique point. Then 
\begin{align*}
\Sigma := \bigsqcup_{d\in D} \Sigma_d \cong i\fg^\ast/G
\end{align*}
is a cross-section of all $\Ad^\ast(G)$-orbits. \\
Furthermore, for each $d\in D$ there exists a decomposition 
\begin{align*}
i\fg^\ast = V_{S(d)}\oplus V_{T(d)}
\end{align*}
as a direct sum of vector spaces and a birational, non-singular, surjective map
\begin{align*}
	\psi_d\colon \Sigma_d \times V_{S(d)} \rightarrow U_d
\end{align*}						
such that for each $l\in\Sigma_d$ its orbit is given by $\cO_l=\psi_d\left(l,V_{S(d)}\right)$.
\end{theorem}
\begin{remark}\label{rem:HlL2Rdn}
For $d\in D$ we know $\cH_l \cong L^2\left(\RR^{d_n/2}\right)$ for all $l\in \Sigma_d$, where $d_n=\dim\cO_l$ for all $l\in \Sigma_d$.
\end{remark}

Now, we collect the ingredients and underlying concepts of the main statements starting at the level of nilpotent Lie algebras. These details will not only be presented as background material but will be crucial for our own results.

\begin{lemma}[see {\cite[Kirillov's Lemma 1.1.12]{corgre}}] \label{lem:KirillovL}
Let $\fg$ be a non-abelian nilpotent Lie algebra whose center $\fz(\fg)=\RR Z$ is one-dimensional. Then $\fg$ can be written as
\begin{align*}
	\fg=\RR Z \oplus \RR Y \oplus \RR X \oplus \fw = \RR X \oplus \fg_0,
\end{align*}
a vector space direct sum with a suitable subspace $\fw$. Furthermore, $[X,Y]=Z$ and 
	$\fg_0= \RR Y \oplus \RR Z \oplus \fw$ is the centralizer of $Y$ and an ideal.
\end{lemma}

In order to study the coadjoint orbits we start with

\begin{lemma}[see {\cite[Lemma 1.3.2]{corgre}}]\label{lem:radeven}
	For $l\in i\fg^\ast$ we define the bilinear form $B_l(X,Y)=l(\lbrack X,Y\rbrack )$ on $\fg$.  Then the radical 
	\begin{align}\label{eq:Defrl}
		\fr_l:=\{X\in\fg : B_l( X,Y)=0 \; \forall\; Y\in\fg\}=\{X\in\fg : \ad^\ast(X)l=0\}
	\end{align}
	has even codimension in $\fg$. Hence coadjoint orbits are of even dimension. \\
	They are actually symplectic manifolds with the non-degenerate skew symmetric 2-form $\omega(l')\in\Lambda^2(T^*_{l'}\cO_l)$ such that $\omega(l')(-(\ad^\ast X) l', -(\ad^\ast Y) l') = l'(\lbrack X,Y\rbrack)$, $l'\in\cO_l$. Note that $\omega$ is $\Ad^\ast(G)$-invariant.
\end{lemma}

Now, we are interested in how we can define an irreducible unitary representation of $G$ given an element $l\in i\fg^\ast$ (with Theorem \ref{thm:BIJ} in mind).

\begin{definition}
A \emph{polarizing subalgebra} for $l\in i\fg^\ast$ is a subalgebra $\fm\subset\fg$ that is a maximal isotropic subspace for the bilinear form $B_l:\fg\times\fg\to i\RR$.\\ 
They are also called \emph{maximal subordinate subalgebras} for $l$.
\end{definition}

\begin{proposition}[see {\cite[Proposition 1.3.3]{corgre}}]
Let $\fg$ be a nilpotent Lie algebra and let $l\in i\fg^\ast$. Then there exists a polarizing subalgebra for $l$.
\end{proposition}

Now, for $l\in i\fg^\ast$ choose a polarizing $\fm$ and let $M=\exp\fm$. Then $\chi_l(\exp Y)=e^{2\pi l(Y)}$ is a one-dimensional representation of $M$ since $l(\lbrack \fm,\fm\rbrack)=0$. Hence, we can define $$\sigma_l:=\Ind_M^G(\chi_l).$$
More precisely, 
\begin{align*}
\cH_l=\{f:G\to \CC \text{ measurable} : f(mg)=\chi_l(m)f(g)\;\forall m\in M \text{ and } \int\limits_{M\backslash G} \|f(g)\|^2 d\dot{g}<\infty\}
\end{align*}
and $$(\sigma_l(x)f)(g)=f(gx)\quad \forall \; x\in G, f\in\cH_l.$$

With this construction one can prove the bijection $\hat{G}\cong i\fg^\ast/G$.

The proof is by induction on  the dimension of $G$. The inductive step is based on the following statement.

\begin{proposition}[see {\cite[Proposition 1.3.4]{corgre}}]\label{prop:cases}
Let $\fg_0$ be a subalgebra of codimension 1 in a nilpotent Lie algebra $\fg$, let $l\in i\fg^\ast$, and let $l_0=l\big|_{\fg_0}$. Let $\fr_l$ be the radical defined in Equation (\ref{eq:Defrl}). Then there are two mutually exclusive possibilities:
\begin{itemize}
\item \textbf{Case I} characterized by any of the following equivalent properties:
	\begin{enumerate}
		\item[(i)] $\fr_l\nsubseteq\fg_0$;
		\item[(ii)] $\fr_l \supset \fr_{l_0}$;
		\item[(iii)] $\fr_{l_0}$ of codimension 1 in $\fr_{l}$.
	\end{enumerate}
	In this case, if $\fm$ is a polarizing subalgebra for $l$, then $\fm_0=\fm\cap\fg_0$ is a polarizing subalgebra for $l_0$; $\fm_0$ is of codimension 1 in $\fm$ and $\fm=\fr_l + \fm_0$.
\item \textbf{Case II} characterized by any of the following equivalent properties:
	\begin{enumerate}
		\item[(i)] $\fr_l \subset \fg_0$;
		\item[(ii)] $\fr_l \subset \fr_{l_0}$;
		\item[(iii)] $\fr_l$ of codimension 1 in $\fr_{l_0}$.
	\end{enumerate}
In this case, any polarizing subalgebra for $l_0$ is also polarizing for $l$.
\end{itemize} 
\end{proposition}

Even though this is a rather technical result its significance becomes clearer in the next statements since we also know how the irreducible representations and the orbits of $G$ and $G_0$ are connected in these  two cases.

\begin{theorem}[see {\cite[Theorem 2.5.1]{corgre}}]\label{thm:sigmaOg0}
Let the notation be as above. Let $p: i\fg^\ast\to i\fg_0^\ast$ be the canonical projection and $G_0=\exp(\fg_0)$.
\begin{enumerate}
	\item[(i)] In Case I, where $\fr_l\nsubseteq\fg_0$, we have 
		\begin{align*}
			\sigma_{l_0}\cong\sigma_l\big|_{G_0} \quad\text{and}\quad p:\cO_l\to\cO_{l_0}:=\Ad^\ast(G_0)l_0 \text{ is a bijection}
		\end{align*}
		(see Figure \ref{fig:OG0G}).
	\item[(ii)] In Case II, where $\fr_l\subset\fg_0$, we have
		\begin{align*}
			\sigma_l\cong\operatorname{Ind}_{G_0}^G(\sigma_{l_0}), \quad
			p(\cO_l)=\bigsqcup_{t\in\RR} (\Ad^\ast\exp tX)\cO_{l_0} \quad\text{and}\quad \cO_l=p^{-1}(p(\cO_l)),
		\end{align*}
		where $X$ is any element such that $\fg=\RR X\oplus \fg_0$.
\end{enumerate}
\end{theorem}

\begin{figure}[ht]\centering
      \includegraphics[width=.6\linewidth]{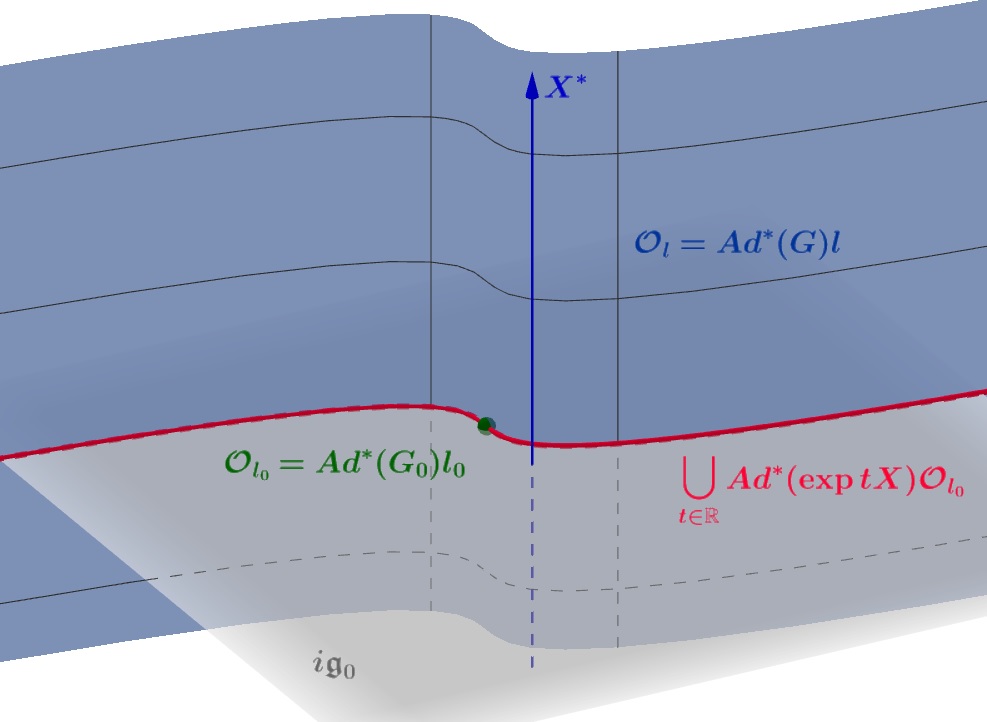}
      \caption{Orbits  of $G_0$ and $G$ in Case II of Theorem \ref{thm:sigmaOg0}} \label{fig:OG0G}    
\end{figure}
In order to nicely formulate the statements about the estimate of matrix coefficients in Section \ref{sec:wf=ac} we introduce the following notation:
\begin{definition}\label{def:gprec}
Let $N$ be a nilpotent, connected, simply connected, nilpotent Lie group with  Lie algebra $\fn$ and fix an inner product on $\fn$. Then for a nilpotent Lie algebra $\fg$ we write $(\fg,\langle ,\rangle_\fg)\prec(\fn,\langle ,\rangle_{\fn})$ if and only if $\fg$ can occur in the induction process of $\fn$ using the two cases of Theorem \ref{thm:sigmaOg0}, i.e. can be obtained via passing to a quotient by a central element or taking the subalgebra of co-dimension 1 given by Kirillov's Lemma \ref{lem:KirillovL}, and the inner product on $\fg$ is the one it inherits from $\fn$.
\end{definition} 

We end this section with two technical lemmas we will use in the next section.

The first lemma concerns the transition maps between different charts of $G$:\\
Let $X_0\in\fg\setminus\{0\}$. Given an inner product in $\fg$ we consider the  orthogonal decomposition $\fg=\RR X_0 \oplus V$ as vector spaces and define
\[
\beta_{X_0}:\Abb{\RR X_0 \oplus V}{G}{tX_0+v}{\exp(tX_0)\exp(v)}
\]
This is clearly a local diffeomorphism
around $0\in\fg$ and we can scale our scalar product on $\fg$ appropriately such that
\begin{equation}\label{eq:kappa_X_0}
\kappa_{X_0} = \beta_{X_0}^{-1}\circ\exp:B_2(0) \to\fg
\end{equation}
is a well defined smooth transition map, and by compactness of the projective space $P\fg$ this can be done uniformly in the choice of $X_0\in\fg\setminus\{0\}$. Thus for each  $N\in\NN$ the quantity $C_{\fg,N}\coloneq \sup_{\|X_0\|=1}\|\kappa_{X_0}\|_{C^N(B_1(0))}$ is finite since it depends continuously on $X_0\in S^{n-1}$.

\begin{lemma}\label{lem:Ckappa}
Let $\fh\subset \fg$ be a subalgebra of co-dimension 1 or $\fh=\fg/\RR Z$ a quotient with $Z\in\fz(\fg)$ and take compatible inner products on $\fg$ and $\fh$. Then $C_{\fh,N}\leq C_{\fg,N}$ for all $N\in\NN$.
\end{lemma}

\begin{proof}
We start with the case that  $\fh\subset \fg$ is a subalgebra of co-dimension 1. Then the exponential map $\exp^\fh$ on $\fh$ is just the exponential map $\exp^\fg$ of $\fg$ restricted to $\fh$. In particular, for $X_0\in\fh$ with $\|X_0\|=1$ we have $\fh=\RR X_0\oplus V_\fh$ and $\fg=\RR X_0\oplus V_\fh\oplus \fh^\perp$, $V_\fg=V_\fh\oplus \fh^\perp$. Thus, $\beta_{X_0}^\fh=\beta_{X_0}^\fg|_{\fh}$ and therefore $\kappa_{X_0}^\fh=\kappa_{X_0}^\fg|_{\fh}$ and $C_{\fh,N}\leq C_{\fg,N}$.

If $\fh=\fg/\RR Z$ is a quotient with $Z\in\fz(\fg)$ we consider the orthogonal complement $W\subset \fg$ of $\RR Z$ in $\fg$ and the vector space isomorphism $\iota:\fh\to W$ such that $\operatorname{pr}:\fg\to\fh$ corresponds to the orthogonal projection.
On the level of the Lie groups we have $H=G/A$ with $A=\exp(\RR Z)$, and $\exp^\fh(X+\RR Z)=\exp^\fg(\iota(X))A\in H$, $\log^\fh(gA)=\log^\fg(g)+\RR Z$.\\
Now, let $X_0\in W, \|X_0\|=1$, $\overline{X}_0=X_0+\RR Z\in\fh$, and $\fh=\RR \overline{X}_0 \oplus V_\fh$. Then $\beta_{\overline{X}_0}^\fh(t\overline{X}_0+\overline{v})=\beta_{\iota(X_0)}^\fg(\iota(tX_0+v))A$ since $Z\in\fz(\fg)$, and $\kappa_{\overline{X}_0}^\fh(t\overline{X}_0+\overline{v})=\kappa_{\iota(X_0)}^\fg(\iota(tX_0+v))+\RR Z = \operatorname{pr}_\fh\circ\kappa_{\iota(X_0)}^\fg(\iota(tX_0+v))$.
This finishes the proof since the projection $\operatorname{pr}_\fh$ can only reduce the norm of derivatives.
\end{proof}
Another estimate which will be useful, later concerns the maximal operator norm of the adjoint action on a Lie algebra $\fg$ with scalar product $\langle\cdot,\cdot\rangle$:
\[
C_{\Ad,\fg} := \sup_{\|X\|, \|Y\| \leq 1} \|[X,Y]\|
\]
We have once more that this quantity can only decrease in the induction process.
\begin{lemma}\label{lem:Ckappa}
Let $\fh\subset \fg$ be a subalgebra of co-dimension 1 or $\fh=\fg/\RR Z$ a quotient with $Z\in\fz(\fg)$ and take compatible inner products on $\fg$ and $\fh$. Then $C_{\Ad,\fh}\leq C_{\Ad, \fg}$.
\end{lemma}
\begin{proof}
 For subalgebras the statement is obvious. If $\fh= \fg/\RR Z$, then let $W<\fg$ be the orthogonal complement and take $X,Y\in W$, $\|X\|_\fg,\|Y\|_\fg\leq 1$, then
 \[
  \|[X+\RR Z, Y+\RR Z]\|_\fh = \|[X,Y]+\RR Z \|_\fh \leq \|[X,Y]\|_\fg\leq C_{\Ad,\fg}
 \]

\end{proof}

\section{Proof of Theorem \ref{thm:wf=ac}}\label{sec:wf=ac}
Let $G$ be a nilpotent, connected, simply connected Lie group with Lie algebra $\fg$ of dimension $n$ and $\fg^\ast$ its vector space dual. By $\hat{G}$ we denote the unitary dual. It is isomorphic to the space of coadjoint orbits $i\fg^\ast / G$.
Let $(\pi,\cH)$ be a unitary representation of $G$. Then we can write
\begin{eqnarray}\label{eq:Intdecomp}
 \pi \cong \int_{\hat{G}}^\oplus \sigma^{\oplus m(\pi,\sigma)}~d\mu_\pi(\sigma), \quad \cH \cong \int_{\hat{G}}^\oplus \cH_\sigma^{\oplus m(\pi,\sigma)}~d\mu_\pi(\sigma),
\end{eqnarray}
where $m(\pi,\sigma)$ keeps track of the multiplicity of $\sigma$ in $\pi$. We recall that for such a representation the orbital support of $\pi$ is given by
\begin{align*}
	\cO - \supp\pi = \bigcup_{\sigma\in\supp(\pi)} \cO_\sigma \; \subset i\fg^\ast, \quad \supp(\pi)=\supp(\mu_\pi),
\end{align*}
where $\cO_\sigma\subset i\fg^\ast$ is the orbit of the coadjoint action corresponding to $\sigma\in\hat{G}$ under the isomorphism $\hat{G} \;\cong\; i\fg^\ast / G $ (see Theorem \ref{thm:BIJ}).\\
We start by using the structure of nilpotent Lie groups and the unitary representations. By Theorem \ref{thm:AbkParamO} after fixing a strong Malcev basis of $\fg$ we have
\begin{align*}
	\hat{G} \;\cong\; i\fg^\ast / G \;\cong\; \Sigma = \bigsqcup_{d\in D} \Sigma_d \;\subset i\fg^\ast,
\end{align*}
where $\Sigma$ is a cross-section of all $G$-orbits and $\Sigma_d$ is a cross-section of all orbits of a certain type $d\in D$, which, in particular, all have the same dimension. Moreover, the set $D$ is finite.\\
Thus, we can push $\mu_\pi$ forward to a positive measure on $\Sigma$ and obtain 
\begin{eqnarray*}
\pi &\cong & \int^\oplus_\Sigma \sigma_l^{\oplus m(\pi,\sigma_l)} ~d\mu_\pi(l) \\
& = & \bigoplus_{d\in D} \int^\oplus_{\Sigma_d} \sigma_l^{\oplus m(\pi,\sigma_l)} ~d\mu_{\pi}(l) =: \bigoplus_{d\in D}  \pi_d.
\end{eqnarray*}
With this decomposition we have
\begin{eqnarray*}
\wf(\pi) = \bigcup_{d\in D} \wf(\pi_d), \qquad \ac(\cO-\supp \pi) = \bigcup_{d\in D} \ac(\cO-\supp\pi_d) 
\end{eqnarray*}
by Proposition \ref{prop:wfoplus} and the fact that $\operatorname{AC}\left(\bigcup_{i=1}^n S_i\right) = \bigcup_{i=1}^n \operatorname{AC}(S_i)$.\\
Therefore, it suffices to show that
\begin{align}
	\ac(\cO-\supp (\pi_d)) = \operatorname{WF}(\pi_d)\quad \forall\;d\in D. \label{eq:decompD}
\end{align}

From now on we fix $d\in D$ and may assume that all the irreducible representations in the support of $\pi$ are of the form $\sigma_l$ for an $l\in\Sigma_d \subset U_d$, where $U_d\subset i\fg^\ast$ is the set of all $l\in i\fg^\ast$ such that its orbit $\cO_l=\Ad^\ast(G)l$ is of type $d$ (see Theorem \ref{thm:AbkParamO}).

Our strategy in the proof of \eqref{eq:decompD} is to prove both inclusions separately in the following two subsections. In both cases we begin with single matrix coefficients $m^{\sigma}_{u,v}(g)=\langle \sigma(g)u,v\rangle$, $\sigma\in\hat{G}$ of type $d$. For the inclusion $\ac(\cO-\supp (\pi))\subset \wf(\pi)$ we find vectors $u,v\in\cH_\sigma$ such that the Fourier transform $\cF(m^{\sigma}_{u,v})$ is bounded from below close to the corresponding orbit $\cO_\sigma$ (see Propositions \ref{prop:ReMcoeffsgeqSimpl} and \ref{prop:ReMcoeffsgeq} in Subsection \ref{subsec:1stIncl}).
For the other inclusion $\wf(\pi)\subset \ac(\cO-\supp(\pi))$ we show that far away from the orbit $\cO_\sigma$ the Fourier transform of all matrix coefficients $m^{\sigma}_{u,v}$ is rapidly decaying (see Proposition \ref{prop:detaO} in Subsection \ref{subsec:2ndIncl}).
Since in both statements the constants can be chosen uniformly for all representations $\sigma\in\hat{G}$ we can then use them to show the desired estimates for the matrix coefficient $m^{\pi}_{u,v}(g)=\langle \pi(g)u,v\rangle=\int^\oplus_{\Sigma_d} m^{\sigma_l}_{u_l,v_l}(g)~d\mu_\pi(l)$ (with corresponding $u_l,v_l\in \cH_l^{\oplus m_l}$) which imply the relation of $\ac(\cO-\supp(\pi))$ and $\wf(\pi)$.

\subsection{Proof of the Inclusion \texorpdfstring{$\ac(\cO-\supp (\pi))\subset \wf(\pi)$}{AC subset WF}}\label{subsec:1stIncl}

For the first inclusion we use Lemma~\ref{lem:HoweKritsPhi}  which states in our setting with $n=\dim\fg$:
\begin{align}
\xi\notin \wf(\pi) \quad \quad \Leftrightarrow \quad \quad
 \exists\; e\in U\subset G, \xi\in V\subset i\fg^\ast \; \forall\; \phi\in C_c^\infty(U) \;\forall\; N>n \; \exists\; C_N>0: \nonumber\\
 |\cF(\langle \pi(\bullet)u,v\rangle \phi)(t \eta)| \leq C_N \|\phi\|_{W^{N+n,1}} \|u\|\|v\|t^{-N} \quad \text{ for } t \gg 0,\; \eta\in V,\; u,v\in\cH, \label{eq:HKritM}
\end{align}
where the constants $C_N$ may be chosen independent of both $\eta\in V$ and $u,v\in\cH$.

As mentioned above we need to find matrix coefficients whose Fourier transform is bounded from below. We use the standard notation $\langle\zeta\rangle := \sqrt{1+\|\zeta\|^2}$.
\begin{proposition}\label{prop:ReMcoeffsgeqSimpl}
Fix an inner product on $\fg$.
There exist $C>0, \varepsilon>0$ such that for all $\zeta\in U_d\subset i\fg^\ast$ we can find vectors $u_\zeta, v_\zeta \in\cH_\zeta^\infty$ with $\|u_\zeta\|=\|v_\zeta\|=1$ that depend measurably on $\zeta$ (i.e. the resulting map $U_d\to L^2(\RR^{d_n/2})\cong\cH_\zeta$ is measurable) such that for all $\eta\in U_d$ with $\|\eta-\zeta\|<\frac{1}{C}$ the following estimate holds for all non-negative $\phi\in C_c^\infty(B_{\frac 1C \langle \zeta \rangle^{-1+\varepsilon}}(0))$:
\begin{align*}
	\re\left( \int_\fg \langle \sigma_\zeta(\exp(X))u_\zeta,v_\zeta \rangle \phi(X) e^{-2\pi\eta(X)}~dX \right) \geq \frac 1C \langle\zeta\rangle^{-1/2}\cdot\int_\fg \phi(X)~dX \geq 0.
\end{align*}
\end{proposition}

Before we can begin with the proof however, we will need to restate this proposition in a more detailed version (see Proposition~\ref{prop:ReMcoeffsgeq}). This is necessary since we want to prove it by induction over $\dim(\fg)$ and need a more detailed induction statement for this. 
We use the notation introduced in Definition \ref{def:gprec} to specify the dependencies of the occurring constants.
The proof of Proposition~\ref{prop:ReMcoeffsgeq} will be based on the distinction of cases for subalgebras of codimension 1 as in Theorem~\ref{thm:sigmaOg0}.
We therefore distinguish the following cases:
\begin{itemize}
\item[i)] If $\zeta(Z)=0$ for some nonzero $Z\in\fz(\fg)$, we consider $\overline{\fg}=\fg/(\RR\cdot Z)$, $\overline{\zeta}=\operatorname{pr}_{i\overline{\fg}^\ast}(\zeta)$ and find that $\sigma_\zeta|_{\overline{G}}\cong \sigma_{\overline{\zeta}}$ and $\cH_\zeta\cong \cH_{\overline{\zeta}}$ analogously to Case I of Theorem~\ref{thm:sigmaOg0}. Thus, we can use for $\sigma_\zeta$ the same vectors that the induction hypothesis applied to $\sigma_{\overline\zeta}$ gives us and check the desired estimates. 

\item[ii)] If $\fz(\fg)=\RR\cdot Z$ and $\zeta(Z)\neq 0$, Kirillov's Lemma \ref{lem:KirillovL} gives us a subalgebra $\fg_0$ to which we apply the induction hypothesis. Writing $\zeta_0=\operatorname{pr}_{i\fg_0^\ast}(\zeta)$ Theorem~ \ref{thm:sigmaOg0} tells us that $\sigma_\zeta=\operatorname{Ind}_{G_0}^G(\sigma_{\zeta_0})$ and this identification allows us to construct the desired vectors $u_\zeta,v_\zeta \in \cH_\zeta$ from two vectors $u_{\zeta_0},v_{\zeta_0}\in \cH_{\zeta_0}$ that are obtained from the induction hypothesis.
However, two difficulties arise: Ideally we would like to take distributional vectors in $\cH_\zeta^{-\infty}$ which we then have to approximate to find a suitable vector in $\cH_\zeta$. Furthermore, to estimate the Fourier transform of the corresponding matrix coefficient we use a chart $\fg\to G$ resulting from the decomposition $\fg=\fg_0\oplus \RR X$ given by the Kirillov Lemma. In order to change to the desired chart $\exp:\fg\to G$ we require further estimations. For these we need an upper bound of the $C^1$-norm of the matrix coefficients which is also added to our second formulation of the proposition.
\end{itemize}
\begin{proposition}\label{prop:ReMcoeffsgeq}
Let $N$ be a nilpotent, connected, simply connected Lie group with  Lie algebra $\fn$ and fix an inner product on $\fn$ such that \eqref{eq:kappa_X_0} is well defined.
Let
\[
 C_0\geq \max(10\pi, e^{C_{\Ad,\fn}}, C_{\fn, 1}, C_{\fn, 2}, C_{\Ad,\fn}+2\pi )
\]
and $\varepsilon>0$ such that $\varepsilon 2^{\dim(\fn)} \leq \frac 12$.

Then for any  $n\leq \dim(\fn)$ and all nilpotent, connected, simply connected Lie groups $G$ with Lie algebra $(\fg,\langle ,\rangle_\fg)\prec(\fn,\langle ,\rangle_{\fn})$ and $\dim\fg=n$,
 and all $\zeta\in U_d\subset i\fg^\ast$ we can find vectors $u_\zeta, v_\zeta \in\cH_\zeta^\infty$, with
 $\|u_\zeta\|=\|v_\zeta\|=1$ that depend measurably on $\zeta$ such that the following estimates hold:
For the matrix coefficient $m_{u_\zeta,v_\zeta}(X)\coloneq \langle \sigma_\zeta(\exp(X))u_\zeta,v_\zeta \rangle$ we have
 \begin{equation}\label{eq:regularity_matrix_coeff}
 	\|m_{u_\zeta,v_\zeta}\|_{C^1(B_{C_0^{-n}}(0))}\leq C_0^{2n} \langle \zeta \rangle
 \end{equation}
Furthermore we have for all $\eta\in U_d$ with $\|\eta-\zeta\|<\frac{1}{C_0}$ the following estimate for all non-negative $\phi\in C_c^\infty(B_{C_0^{-3^n}\langle \zeta \rangle^{-1+\varepsilon}}(0))$:
\begin{equation}\label{eq:lower_bound}
	\re\left( \int_\fg m_{u_\zeta,v_\zeta}(X) \phi(X) e^{-2\pi\eta(X)}~dX \right) \geq C_0^{-3^n}\langle\zeta\rangle^{ -\varepsilon2^n}\int_\fg \phi(X)~dX \geq 0.
\end{equation}

\end{proposition}
\begin{proof}

We prove this statement by induction on $n=\dim\fg$. If $n=1,2$, the group is abelian. In this case the irreducible unitary representations are one-dimensional, i.e. $\sigma_\zeta(g)=e^{2\pi\zeta(\log g)}$, $\cH_\zeta=\CC$. We choose $u_\zeta=v_\zeta=1$ and compute $|d_X m_{u_\zeta,v_\zeta}(X)|=2\pi\|\zeta\|$, thus we have proven \eqref{eq:regularity_matrix_coeff} for $n=1,2$ since $C_0>10\pi$ the estimate is largely fulfilled.\\
For the estimate of the integral we have
\begin{align*}
\re&\left( \int_\fg \langle \sigma_\zeta(\exp(X))u_\zeta,v_\zeta \rangle \phi(X) e^{-2\pi\eta(X)}~dX \right)
	= \re\left( \int_\fg e^{2\pi(\zeta-\eta)(X)}\phi(X)~dX \right) \\
	&= \int_\fg \re\left(e^{2\pi(\zeta-\eta)(X)}\right)\phi(X)~dX
	= \int_\fg \cos(2\pi i(\eta-\zeta)(X))\phi(X)~dX
	 \geq \frac{1}{C_0} \int_\fg \phi(X)~dX
\end{align*}
since $|i((\eta-\zeta)(X)|\leq \|\eta-\zeta\|\cdot\|X\|\leq \frac{1}{C_0^4}$ for $X\in \supp \phi$ and since $C_0>2\pi$. We have thus shown \eqref{eq:lower_bound} for $n=1,2$.

Now we assume $n=\dim\fg\geq 3$. We will distinguish between the two cases following Theorem~\ref{thm:sigmaOg0}.

\textbf{Case I: $\zeta(Z)=0$ for an $Z\in\fz(\fg)$.} Without loss of generality we may assume $\|Z\|=1$. We can choose the orthogonal complement $W<\fg$ such that $\fg=W\oplus\RR Z$.
Then $\overline{\fg}=\fg/(\RR\cdot Z)$ is isomorphic to $W$ and has a well-defined Lie algebra structure given by $\lbrack v+\RR Z, w+\RR Z\rbrack~=~\lbrack v,w\rbrack_\fg +\RR Z$ since $Z\in\fz(\fg)$.

On $\overline{\fg}$ we use the inner product induced from the one we fixed on $\fg$.
Using the corresponding inner products on $i\fg^\ast$ and $i\overline{\fg}^\ast$ we also obtain an orthogonal decomposition $i\fg^\ast=iW^\ast\oplus i\RR Z^* \cong i\overline{\fg}^\ast \oplus i\RR Z^*$ with $\|Z^*\|=1$.

Note that $i\overline{\fg}^\ast$ is $\Ad^\ast(G)$-invariant (again due to  $Z\in\fz(\fg)$). As we assumed $\zeta(Z)=0$, we can identify $\zeta$ with an element $\overline{\zeta}\in i\overline{\fg}^\ast$.
Let $\eta=\overline{\eta}+irZ^*\in i\fg^\ast=i\overline{\fg}^\ast\oplus i\RR Z^*$. By assumption $|r|=|(\eta-\zeta)(Z)|\leq \frac{1}{C_0}$.

The induction hypothesis also gives us normalized vectors $u_{\overline{\zeta}}, v_{\overline{\zeta}}\in\cH_{\overline{\zeta}}^\infty$.
By Theorem~\ref{thm:sigmaOg0}~(i) $\cH_{\overline{\zeta}}\cong \cH_\zeta$ and $\sigma_{\overline{\zeta}}\circ P\cong \sigma_\zeta$  with the projection $P:G\to \overline{G}$.
Thus, we obtain corresponding vectors $u_\zeta=u_{\overline{\zeta}}\in\cH_\zeta^\infty$, $v_\zeta=v_{\overline{\zeta}}\in\cH^\infty_\zeta$ and compute
\begin{align*}
&d_{t}m_{u_\zeta,v_\zeta}(\overline{X}+tZ)=0, \\
&|\partial_{\overline{X}_0} m_{u_\zeta,v_\zeta}(\overline{X}+tZ)|
= |\partial_{\overline{X}_0} m_{u_{\overline{\zeta}},v_{\overline{\zeta}}}(\overline{X})|
\leq C_0^{2(n-1)}\langle \overline{\zeta}\rangle
\leq C_0^{2n}\langle \zeta \rangle, \quad \text{ for } \overline{X}_0\in\overline{\fg} \text{ with } \|\overline{X}_0\|=1,
\end{align*}
and we have proven the induction step for the regularity estimate \eqref{eq:regularity_matrix_coeff}.

For the lower bounds \eqref{eq:lower_bound} we have
\begin{align*}
	R &:= \re\left( \int_\fg \langle \sigma_\zeta(\exp(X))u_\zeta,v_\zeta \rangle \phi(X) e^{-2\pi\eta(X)}~dX \right) \\
	&= \re\left( \int_{\overline{\fg}} \int_\RR \langle \sigma_\zeta(\exp(\overline{X}+tZ))u_\zeta,v_\zeta \rangle \phi(\overline{X}+tZ) e^{-2\pi\eta(\overline{X}+tZ)}~d\overline{X}~dt \right) \\
	&= \re\left( \int_{\overline{\fg}} \int_\RR \langle \sigma_\zeta(\exp(\overline{X})\exp(tZ))u_\zeta,v_\zeta \rangle \phi(\overline{X}+tZ) e^{-2\pi(\overline{\eta}(\overline{X})+irZ^*(tZ))}~d\overline{X}~dt \right) \\
	&= \re\left( \int_{\overline{\fg}} \int_\RR \langle \sigma_{\overline{\zeta}}(\exp(\overline{X}))u_{\overline{\zeta}},v_{\overline{\zeta}} \rangle \phi(\overline{X}+tZ) e^{-2\pi(\overline{\eta}(\overline{X})+irt)}~d\overline{X}~dt \right) \\
	&= \int_\RR \cos(-2\pi rt) \re\left( \int_{\overline{\fg}} \langle \sigma_{\overline{\zeta}}(\exp(\overline{X}))u_{\overline{\zeta}},v_{\overline{\zeta}} \rangle \phi(\overline{X}+tZ) e^{-2\pi\overline{\eta}(\overline{X})}~d\overline{X}\right) \\
	&  \quad \quad - \sin(-2\pi rt)\im\left( \int_{\overline{\fg}} \langle \sigma_{\overline{\zeta}}(\exp(\overline{X}))u_{\overline{\zeta}},v_{\overline{\zeta}} \rangle \phi(\overline{X}+tZ) e^{-2\pi\overline{\eta}(\overline{X})}~d\overline{X}\right)~dt\\
	& \geq  \int_\RR \frac 12\re\left( \int_{\overline{\fg}} \langle \sigma_{\overline{\zeta}}(\exp(\overline{X}))u_{\overline{\zeta}},v_{\overline{\zeta}} \rangle \phi(\overline{X}+tZ) e^{-2\pi\overline{\eta}(\overline{X})}~d\overline{X}\right) \\
	&\quad \quad - \frac{2\pi}{C_0}\cdot C_0^{-3^n}\langle \zeta\rangle^{-1+\varepsilon} \left| \int_{\overline{\fg}} \langle \sigma_{\overline{\zeta}}(\exp(\overline{X}))u_{\overline{\zeta}},v_{\overline{\zeta}} \rangle \phi(\overline{X}+tZ) e^{-2\pi\overline{\eta}(\overline{X})}~d\overline{X} \right| ~dt.
\end{align*}
Since $|2\pi rt|\leq \frac{2\pi}{C_0}\cdot C_0^{-3^n}\langle \zeta\rangle^{-1+\varepsilon}$
for $\overline{X}+tZ\in\supp (\phi)$  and since $\cos(2\pi rt)>1/2$ by the choice of $C_0$. Using the induction
hypothesis as well as $\left|\langle \sigma_{\overline{\zeta}}
(\exp(\overline{X}))u_{\overline{\zeta}},v_{\overline{\zeta}} \rangle e^{-2\pi\overline{\eta}(\overline{X})} \right|=1$ we get
\begin{align*}
	R & \geq  \int_\RR \left(\frac 12 C_0^{-3^{n-1}}\langle \zeta\rangle^{-\varepsilon 2^{n-1}} \int_{\overline \fg} \phi(\overline{X}+tZ) ~d\overline{X} - \frac{2\pi}{C_0}\cdot C_0^{-3^n}\langle \zeta\rangle^{-1+\varepsilon}   \int_{\overline{\fg}} \phi(\overline{X}+tZ) ~d\overline{X} \right) ~dt\\
	&\geq \left(\frac {C_0^3}{2}  - \frac{2\pi}{C_0} \right)C_0^{-3^n}\langle \zeta\rangle^{-\varepsilon 2^{n-1}}\int_\fg \phi(X)dX\\
	&\geq C_0^{-3^n}\langle \zeta\rangle^{-\varepsilon 2^{n}}\int_\fg \phi(X)dX
\end{align*}
Here, we used crucially that $-1+\varepsilon < -\varepsilon 2^n$ for all $n\leq \dim \fn$. We have thus shown the induction step for the lower bound \eqref{eq:lower_bound} and completed the induction for the case I.

\textbf{Case II: $\fz(\fg)=\RR\cdot Z$ and $\zeta(Z)\neq 0$.}
Kirillov's Lemma \ref{lem:KirillovL} gives us $X,Y\in\fg$ and an ideal $\fg_0\subset\fg$ with $\fg=\RR X \oplus \fg_0$ and $[X,Y]=Z$. We may choose $X$ such that the decomposition is orthogonal and $\|X\| = 1$.
Furthermore, $X\notin\fr_l$ and we are in Case II of Proposition~\ref{prop:cases} and Theorem \ref{thm:sigmaOg0} with $G_0=\exp(\fg_0)\subset G$ a normal subgroup.  We define a chart around $e\in G$ via
\begin{align}\label{eq:defbeta}
\beta:\fg=\fg_0 \oplus\RR X \to G, \quad X_0+tX\mapsto\exp(X_0)\exp(tX).
\end{align}
Let $p:i\fg^\ast\to i\fg_0^\ast$ be the canonical projection and $X^*\in\fg^\ast$ defined by $X^\ast(X_0+tX) = t$. Let us write $\zeta=\zeta_0+izX^\ast,\eta=\eta_0+irX^\ast \in \ker(p)^\perp\oplus \ker(p)$. Then by assumption $|z-r|= |(\zeta-\eta)(X)|\leq \frac{1}{C_0}$.

By Theorem \ref{thm:sigmaOg0}, we know $\sigma_\zeta\cong \Ind_{G_0}^G(\sigma_{\zeta_0})$ with $\cH_\zeta\cong L^2(A,\cH_{\zeta_0})$, where $A=\exp(\RR\cdot X)$.
Thus, if we regard $u$ and $v$ as elements of $ L^2(A,\cH_{\zeta_0})$ and $\tilde{u},\tilde{v}:G\to\cH_{\zeta_0}$ the corresponding left-$G_0$-equivariant functions we have for $g_0\in G_0$ and $a\in A$:

\begin{align*}
	& \langle \sigma_\zeta(g_0 a)u,v\rangle_{\cH_\zeta} = \int_A \langle \lbrack\sigma_\zeta(g_0 a)u\rbrack (b), v(b) \rangle_{\cH_{\zeta_0}}~db \quad \text{ and } \\
	& \lbrack\sigma_\zeta(g_0a)\tilde{u}\rbrack (b) = \tilde{u}(bg_0 a)=\tilde{u}(bg_0b^{-1}b a)=\sigma_{\zeta_0}(bg_0b^{-1})\tilde{u}(ba)
	\end{align*}
since $b^{-1}g_0b\in G_0$ as $\fg_0$ is an ideal. This gives us $\lbrack\sigma_\zeta(g_0a)u\rbrack (b) = \sigma_{\zeta_0}(bg_0b^{-1})u(ba)$.

Furthermore, the induction hypothesis gives us, normalized vectors $u_{\zeta_0}, v_{\zeta_0} \in \cH_{\zeta_0}^\infty$, depending measurably on $\zeta_0$.
In order to find the suitable vectors $u_\zeta, v_\zeta\in\cH^\infty_\zeta$ we begin with a cut-off function
$\chi\in C_c^\infty(\RR)$ with $0\leq\chi\leq 1$, $\chi=1$ on $[-1/4,1/4]$ and $\supp\chi \subset[-1,1]$ and $\|\chi\|_{L^2(\RR)} = 1$ and $\|\chi'\|_\infty\leq 2$. Define for $k_1, k_2>1$
\begin{align*}
 u_{\zeta, k_1}(a) &:= \sqrt{k_1}\chi(k_1 X^*(\log(a))) e^{2\pi i z X^*(\log(a))} \otimes u_{\zeta_0} \in C_c^\infty(A,\cH_{\zeta_0}^\infty)\\
 v_{\zeta, k_2}(a) &:= \sqrt{k_2} \chi(k_2 X^*(\log(a))) \otimes v_{\zeta_0}\in   C_c^\infty(A,\cH_{\zeta_0}^\infty)
\end{align*}
With these definitions we have $\|u_{\zeta,k_1}\|_{\mathcal H_\zeta} =\|v_{\zeta, k_2}\|_{\mathcal H_\zeta} = 1$ for all $k_1,k_2>1$. We think about these two families of vectors as smooth approximations of a vector that has a Dirac delta at $\log(a)=0$ and we will late choose $k_1,k_2$ large enough to get the good lower bounds \eqref{eq:lower_bound} but not too large  in order to still have the regularity estimates \eqref{eq:regularity_matrix_coeff}.

Let us introduce the notation
\[
M_{u_{\zeta,k_1},v_{\zeta,k_2}}(\tilde X) :=\langle\sigma_\zeta(\beta(\tilde X))u_{\zeta,k_1},v_{\zeta,k_2} \rangle
\]
and compute for $\tilde X = X_0 +tX\in \fg$
\begin{align*}
 M_{u_{\zeta,k_1},v_{\zeta,k_2}}(\tilde X)&:=\langle\sigma_\zeta(\beta(\tilde X))u_{\zeta,k_1},v_{\zeta,k_2} \rangle\\
 &=\int_A \langle \sigma_{\zeta_0}(b\exp(X_0)b^{-1})u_{\zeta,k_1}(be^{tX}),v_{\zeta,k_2}(b) \rangle~db\\
 &= \sqrt{k_1k_2}\int_A \langle \sigma_{\zeta_0}(b\exp(X_0)b^{-1})u_{\zeta_0},v_{\zeta_0} \rangle \chi(k_1 X^*(\log(be^{tX})))e^{2\pi izX^*(\log(be^{tX}))}\cdot \\
 &\qquad\qquad\chi(k_2 X^*\log b) ~db\\
 &= \sqrt{k_1k_2} \int_\RR \langle \sigma_{\zeta_0}(e^{sX}\exp(X_0)e^{-sX})u_{\zeta_0},v_{\zeta_0} \rangle \chi(k_1(s+t))e^{2\pi iz(s+t)}  \chi(k_2 s) ~ds\\
 &= \sqrt{k_1k_2} \int_\RR  m_{u_{\zeta_0}, v_{\zeta_0}}(\Ad(e^{s X})X_0) \chi(k_1(s+t))e^{2\pi iz(s+t)}  \chi(k_2 s) ~ds.
\end{align*}
In order to prove the lower bounds in the sense of \eqref{eq:lower_bound}, note that by definition of $\chi$, $\phi_{k_2}(s):= \frac{k_2}{\|\chi\|_{L^1}}\chi(k_2 s)$ is a sequence of smooth function converging towards the Dirac distribution at $s=0$ for $k_2\to\infty$. More quantitatively, by the mean value theorem we get for any $f \in C^\infty(\RR)$
\[
 \left|\int_\RR f(s)\phi_{k_2}(s) ds - f(0)\right|\leq \frac{2}{k_2} \sup_{s\in[-1/k_2,1/k_2]} | f'(s)|
\]
Using once more the induction hypothesis on the regularity of $m_{u_{\zeta_0},v_{\zeta_0}}$ we deduce that for $s\in[-1,1]$ and $X_0+tX\in B_{C_0^{-3^n}\langle\zeta\rangle^{-1+\varepsilon}}(0)$
\[
 \left|\frac{d}{ds} m_{u_{\zeta_0}, v_{\zeta_0}}(\Ad(e^{s X})X_0) \chi(k_1(s+t))e^{2\pi iz(s+t)}\right|\leq C_{\Ad, \fg} C_0^{2(n-1)} \langle\zeta\rangle + k_1\|\chi'\|_\infty + 2\pi |z| \leq C_0^{2n-1}\langle\zeta\rangle + 2k_1
\]
as $C_{\Ad,\fg}\leq C_{\Ad,\fn}\leq C_0$ and $|z|\leq \langle \zeta\rangle$. Consequently we get
\begin{align*}
 \left|\sqrt{\frac{k_2}{k_1}}\frac{1}{\|\chi\|_{L^1}} M_{u_{\zeta,k_1},v_{\zeta,k_2}}(X_0+tX) - m_{u_{\zeta_0}, v_{\zeta_0}}(X_0) \chi(k_1 t)e^{2\pi izt} \right| \leq 2\frac{C_0^{2n-1}\langle\zeta\rangle + 2 k_1}{k_2}
\end{align*}
Recall that since $0\leq \chi\leq 1$, that $\|\chi\|_{L^1} \geq \|\chi\|^2_{L^2}= 1$ and we compute
\begin{align*}
 R&:= \re\left( \int_\fg M_{v_{\zeta,k_1},u_{\zeta,k_2}}(\tilde X) \phi(\tilde X) e^{-2\pi\eta(\tilde X)}~d\tilde X \right) \\
 &\geq \sqrt{\frac{k_1}{k_2}}\re\left( \int_{\fg_0} \int_\RR \frac{1}{\|\chi\|_{L^1}}\sqrt{\frac{k_2}{k_1}} M_{v_{\zeta,k_1},u_{\zeta,k_2}}(X_0+tX) \phi(X_0+tX) e^{-2\pi\eta(X_0+tX)}~dX_0~dt \right)  \\
 &\geq \sqrt{\frac{k_1}{k_2}}\bigg[\re\left( \int_{\fg_0} \int_\RR m_{u_{\zeta_0}, v_{\zeta_0}}(X_0) \chi(k_1 t)e^{2\pi izt} \phi(X_0+tX) e^{-2\pi\eta(X_0+tX)}~dX_0~dt \right) \\
 &\quad\quad\quad -2\frac{C_0^{2n-1}\langle\zeta\rangle + 2k_1}{k_2}\int_{\fg} \phi(\tilde X)d\tilde X\bigg] \\
 \end{align*}
Let us only consider the first term in the brackets and recall that $\eta= \eta_0+irX^*$ and that $\chi(k_1 t)=1$ for $X_0+t X \in \supp \phi$, provided that $k_1 C_0^{-3^n}\langle\zeta\rangle^{-1+\varepsilon} \leq 1/4$. We get
 \begin{align*}
 &\re\left( \int_{\fg_0} \int_\RR m_{u_{\zeta_0}, v_{\zeta_0}}(X_0) \chi(k_1t)e^{2\pi izt} \phi(X_0+tX) e^{-2\pi(\eta_0X_0+irt)}~dX_0~dt \right) \\
& = \int_\RR \cos(2\pi (z-r)t) \re\left( \int_{\fg_0}  m_{u_{\zeta_0}, v_{\zeta_0}}(X_0)  \phi(X_0+tX) e^{-2\pi\eta_0(X_0)}~dX_0\right)  \\
&\quad  -\sin(2\pi (z-r)t) \im\left( \int_{\fg_0}  m_{u_{\zeta_0}, v_{\zeta_0}}(X_0) \phi(X_0+tX) e^{-2\pi\eta_0(X_0)}~dX_0\right)~dt\\
& \geq \int_\RR \frac 12  \re\left( \int_{\fg_0}  \langle \sigma_{\zeta_0}(\exp(X_0))u_{\zeta_0},v_{\zeta_0} \rangle  \phi(X_0+tX) e^{-2\pi\eta_0(X_0)}~dX_0\right)  \\
&\quad  -|2\pi (z-r)t| \left|\int_{\fg_0}  \langle \sigma_{\zeta_0}(\exp(X_0))u_{\zeta_0},v_{\zeta_0} \rangle  \phi(X_0+tX) e^{-2\pi\eta_0(X_0)}~dX_0\right|~dt\\
\end{align*}
Using the lower bound \eqref{eq:lower_bound} for the integral over $\fg_0$ from the induction hypothesis as well as the estimates $|z-r|\leq \|\eta-\zeta\|\leq\frac{1}{C_0}$ as well as $|t|\leq C_0^{-3^n}\langle\zeta\rangle^{-1+\varepsilon}$ for $X_0+tX\in\supp\phi$ we can estimate
\begin{align*}
 &\re\left( \int_{\fg_0} \int_\RR m_{u_{\zeta_0}, v_{\zeta_0}}(X_0) \chi(k_1 t)e^{2\pi izt} \phi(X_0+tX) e^{-2\pi(\eta_0X_0+irt)}~dX_0~dt \right) \\
 &\geq \left(\frac{1}{2}C_0^{-3^{n-1}}\langle\zeta \rangle^{-\varepsilon 2^{n-1}}  - 2\pi C_0^{-3^n-1} \langle\zeta\rangle^{-1+\varepsilon}\right) \int_{\fg_0} \int_\RR \phi(X_0+tX) dX_0 dt.
\end{align*}
Plugging this estimate in the above estimate for $R$ we get
\begin{align*}
 R & \geq \sqrt{\frac{k_1}{k_2}} \left(\frac{1}{2}C_0^{-3^{n-1}}\langle\zeta \rangle^{-\varepsilon2^{n-1}}  - 2\pi C_0^{-3^n-1} \langle\zeta\rangle^{-1+\varepsilon}-2\frac{C_0^{2n-1}\langle\zeta\rangle + 2k_1}{k_2}\right) \int_{\fg}\phi(\tilde X) ~d\tilde X \\
\end{align*}
If we now set the values $k_1 = \frac 14 C_0^{2(n-1)}\langle\zeta\rangle^{1-\varepsilon}$ and $k_2= C_0^{3^n}\langle \zeta\rangle^{1+\varepsilon2^{n-1}}$ we get
\[
 \frac{C_0^{2n-1}\langle\zeta\rangle + 2k_1}{k_2} \leq 2C_0^{-3^{n-1}-1}\langle\zeta\rangle^{-\varepsilon2^{n-1}} \text{ and }\sqrt{\frac{k_1}{k_2}}\geq \frac 12 C_0^{-\frac 12 3^{n} }\langle\zeta\rangle^{-\varepsilon2^{n-1}}
\]
and thus setting $u_\zeta:= u_{\zeta,k_1}$ and $v_\zeta:=v_{\zeta,k_2}$ we have thus shown (recall once more that $-\varepsilon 2^n>-1+\varepsilon$)
\begin{align}
 \re\left( \int_\fg M_{v_\zeta,u_{\zeta}}(\tilde X) \phi(\tilde X) e^{-2\pi\eta(\tilde X)}~d\tilde X   \right) & \geq C_0^{-\frac 52 3^{n-1}}\langle\zeta\rangle^{-\varepsilon2^{n}} \frac{1}{2}\left( \frac{1}{2} - \frac{2\pi}{C_0^{2\cdot 3^{n-1}}} - \frac{2}{C_0} \right)   \int_{\fg}\phi(\tilde X) ~d\tilde X \nonumber\\
 &\geq C_0^{-\frac 52 3^{n-1}-1}\langle\zeta\rangle^{-\varepsilon2^{n}}\int_{\fg}\phi(\tilde X) ~d\tilde X\label{eq:M_lower_bound}
\end{align}

We also compute a bound for the $C^1$-norm  of $M_{u_\zeta, v_{\zeta}}$: For $X_0\in\fg_0$  with $\|X_0\|=1$ as well as for arbitrary $Y_0+t_0 X\in B_{C_0^{-(n-1)}}(0)\subset  \fg_0\oplus \RR X= \fg$ we get
\begin{align*}
\left|\frac{\partial}{\partial{r}}\Big|_{r=0}  M_{u_\zeta,v_\zeta}(Y_0+rX_0+t_0X) \right|
&\leq  \int_\RR  \left|\frac{\partial}{\partial{r}}_{|r=0} m_{u_{\zeta_0},v_{\zeta_0}}\left((\Ad\left(e^{sX}\right)Y_0)+r(\Ad(e^{sX})X_0)\right)\right|\\
&\qquad\qquad\sqrt{k_1 k_2}\chi(k_1(s+t_0))\chi(k_2s)~ds \\
&\leq e^{C_{\Ad,\fg}} \|m_{u_{\zeta_0},v_{\zeta_0}}\|_{C^1(B_{C_0^{-(n-1)}}(0))} \cdot\\
&\qquad\qquad\sqrt{k_1}\|\chi(k_1(\bullet+t_0) )\|_{L^2(\RR)}\| \sqrt{k_2}\chi(k_2\bullet )\|_{L^2(\RR)} \\
&\leq e^{C_{\Ad,\fg}} C_0^{2(n-1)} \langle\zeta\rangle
\end{align*}
Here we applied the induction hypothesis for the regularity of $m_{u_{\zeta_0},v_{\zeta_0}}$. In the remaining direction we have:
\begin{align*}
\Big| \frac{\partial}{\partial{r}}\Big|_{r=0}  &M_{u_\zeta,v_{\zeta,k}}(Y_0+(t_0+r)X) \Big|\\
&\leq \bigg| \int_\RR M_{u_{\zeta_0},v_{\zeta_0}}\left(\Ad\left(e^{sX}\right)Y_0\right) \left(k_1\chi'\left(k_1(s+t_0) \right)e^{2\pi iz(s+t_0)} + \chi\left(k_1(s+t_0\right))2\pi iz e^{2\pi iz(s+t_0)}  \right)\cdot\\
&\qquad\qquad \sqrt{ k_1k_2}\chi(k_2s)ds\bigg| \\
&\leq k_1\|\sqrt{k_1} \chi'(k_1(\bullet+t_0))\|_{L^2} \|\sqrt{k_2}\chi(k_2\bullet)\|_{L^2} +2\pi|z| \|\sqrt{k_1}\chi(k_1(\bullet +t_0 )) \|_{L^2}\|\sqrt{k_2}\chi(k_2\bullet)\|_{L^2} \\
&\leq 4k_1 + 2\pi|z|\\
&\leq C_0^{2(n-1)} \langle\zeta\rangle +2\pi\left(\|\zeta\|+\frac{1}{C_0}\right).
\end{align*}
Putting these estimates together we get since $C_0$ had been chosen large enough:
\begin{align*}
\|M_{u_\zeta,v_{\zeta}}\|_{C^1(B_{C_0^{-(n-1)}}(0))} \leq C_0^{2n-1} \langle \zeta\rangle.
\end{align*}

Now, recall that the matrix coefficients $M_{u_\zeta,v_\zeta}$ are defined via the chart $\beta$ from \eqref{eq:defbeta}, so it remains to transform this back to a matrix coefficient defined with the exponential map in order to complete the induction step.
Thus, we define the transition map $\kappa=\beta^{-1}\circ\exp:\fg\to\fg$ to replace the matrix coefficient $M_{u_\zeta,v_\zeta}(\tilde X)$ by the matrix coefficients $m_{u_\zeta,v_\zeta}(\tilde X)=M_{u_\zeta,v_\zeta}(\kappa(\tilde X))$.
For the $C^1$-norm of these matrix coefficients note that by the mean value theorem and as $C_{\fn, 1}\leq C_0$ we know that $\kappa(B_{C_0^{-n}}(0))\subset B_{C_0^{-(n-1)}}(0)$ and we immediately see with Lemma~\ref{lem:Ckappa} that
\begin{align*}
\|m_{u_\zeta,v_\zeta}\|_{C^1(B_{C_0^{-n}}(0))}\leq \|D\kappa\|_{L^\infty(B_1(0))}  \|M_{u_\zeta,v_\zeta}\|_{C^1(B_{C_0^{-(n-1)}}(0))} \leq C_{\fg,1}\|M_{u_\zeta,v_\zeta}\|_{C^1} \leq C_0^{2n}\langle \zeta \rangle.
\end{align*}

In order to estimate the Fourier transform  we look at the following difference in $X\in \supp\phi$:
\begin{align*}
|m_{u_\zeta,v_\zeta}(\tilde{X} )-M_{u_\zeta,v_\zeta}(\tilde X)|
= |M_{u_\zeta,v_\zeta}(\kappa(\tilde{X})) - M_{u_\zeta,v_\zeta}(\tilde{X})|
\leq \|M_{u_\zeta,v_\zeta}\|_{C^1(B_1(0))} \|\kappa(\tilde{X})-\tilde{X}\|
\end{align*}
by the mean value theorem. If we use the Taylor expansion of $\kappa$ in $0$ we have since $D_0\kappa=\Id_{\fg}$:
$$\|\kappa(\tilde{X})-\tilde{X}\|\leq \|\kappa\|_{C^2}\|\tilde{X}\|^2\leq C_{\fn,2} \operatorname{diam}(\supp(\phi))^2,$$
using Lemma \ref{lem:Ckappa} again. Therefore, we have for all $\tilde{X} \in \supp(\phi)$:
\begin{align*}
|m_{u_\zeta,v_\zeta}(\tilde{X})-M_{u_\zeta,v_\zeta}(\tilde{X})|
&\leq \|M_{u_\zeta,v_\zeta}\|_{C^1}C_{\fn,2} \operatorname{diam}(\supp(\phi))^2\\
&\leq C_0^{2n}\langle\zeta\rangle C_{\fn,2} (C_0^{-3^n}\langle\zeta\rangle^{-1+\varepsilon})^{2}\\
&\leq C_0^{2n -2\cdot3^n +1}\langle\zeta\rangle^{-1 + 2 \varepsilon}
\end{align*}
by our choice of $\supp\phi$. With this we can estimate
\begin{align*}
&\re\left( \int_\fg m_{u_\zeta,v_\zeta}(\tilde{X}) \phi(X) e^{-2\pi\eta(\tilde{X})}~d\tilde{X} \right) \\
&= \re\left( \int_\fg M_{u_\zeta,v_\zeta}(\tilde{X}) \phi(\tilde{X}) e^{-2\pi\eta(\tilde{X})}~d\tilde{X} \right)  + \re\left(  \int_\fg \left(m_{u_\zeta,v_\zeta}(\tilde{X})- M_{u_\zeta,v_\zeta}(\tilde{X})\right) \phi(\tilde{X}) e^{-2\pi\eta(\tilde{X})}~dX\right) \\
&\geq   C_0^{-\frac 52 3^{n-1}-1}\langle\zeta\rangle^{-\varepsilon2^{n}} \int_{\fg}\phi(\tilde X) ~d\tilde X  -  C_0^{2n -2\cdot3^n+1 }\langle\zeta\rangle^{-1 + 2 \varepsilon}\int_\fg \phi(\tilde{X})~d\tilde{X}\\
&\geq \left(C_0^{-\frac 52 3^{n-1} +  3^n -1} - C_0^{2n-3^n+1}\right) C_0^{-3^n}\langle\zeta\rangle^{-\varepsilon2^n}\int_\fg \phi(\tilde{X})~d\tilde{X}\\
&\geq C_0^{-3^n}\langle\zeta\rangle^{-\varepsilon2^n}\int_\fg \phi(\tilde{X})~d\tilde{X}\\
\end{align*}
Here we again use that $-1+\varepsilon < -\varepsilon 2^{n-1}$ for all $n\leq \dim\mathfrak n$ and furthermore that $\left(C_0^{-\frac 52 3^{n-1} +  3^n -1} - C_0^{2n-3^n+1}\right)\geq 1$ since in the induction process we have $n\geq 3$. This is the desired estimate.
\end{proof}

Now we can turn to the desired statement:

\begin{theorem}\label{thm:ACsubWF}
Let $G$ be a nilpotent, connected, simply connected Lie group with Lie algebra $\fg$ and $(\pi,\cH_\pi)$ a unitary representation of $G$. Then
	\begin{eqnarray*}
		\ac(\cO-\supp \pi)\subset \wf(\pi).
	\end{eqnarray*}
\end{theorem}

\begin{proof}
Let $\xi\in \ac(\cO-\supp \pi)$. We may assume without loss of generality that $\|\xi\|=1$. Defining the cones $\cC_\varepsilon := \{\eta \in i\fg^\ast\mid \exists\; t>0:\;\|\xi-t\eta\|<\varepsilon\}$, then for all $\varepsilon >0$ there exists a sequence $(t_m\eta_m)_m\subset \cC_\varepsilon\cap \cO-\supp(\pi)$ with $t_m\to \infty$ and $\eta_m\in B_\varepsilon(\xi)$, $\|\eta_m\|=1$. \\
We now use Theorem \ref{thm:AbkParamO}: For all $m\in\NN$ let $l_m\in\Sigma_d$ be the corresponding element in the cross-section of all orbits of type $d$, i.e. $\cO_{l_m}=\cO_{t_m\eta_m}$.  Then there exists $v_m\in V_{S(d)}$ with $t_m\eta_m=\psi_d(l_m,v_m)$. For $l\in\Sigma_d$ near $l_m$ we define $\zeta_l:=\psi_d(l,v_m)\in\cO_l$ which depends continuously on $l$ (see Figure \ref{fig:Nm}). \\
Now let $C,\varepsilon>0$ be as in Proposition \ref{prop:ReMcoeffsgeqSimpl}and set $\delta=\frac{1}{C}$. Then there exists a neighborhood $N_m\subset\Sigma_d$ of $l_m$ such that $\psi_d(N_m,v_m)\subset B_{\delta}(t_m\eta_m)$ and $\mu_\pi(N_m)>0$ since $l_m\in \cO-\supp(\pi)$ (see also Figure \ref{fig:Nm}).

\begin{figure}[ht] \centering
\includegraphics[width=.5\textwidth]{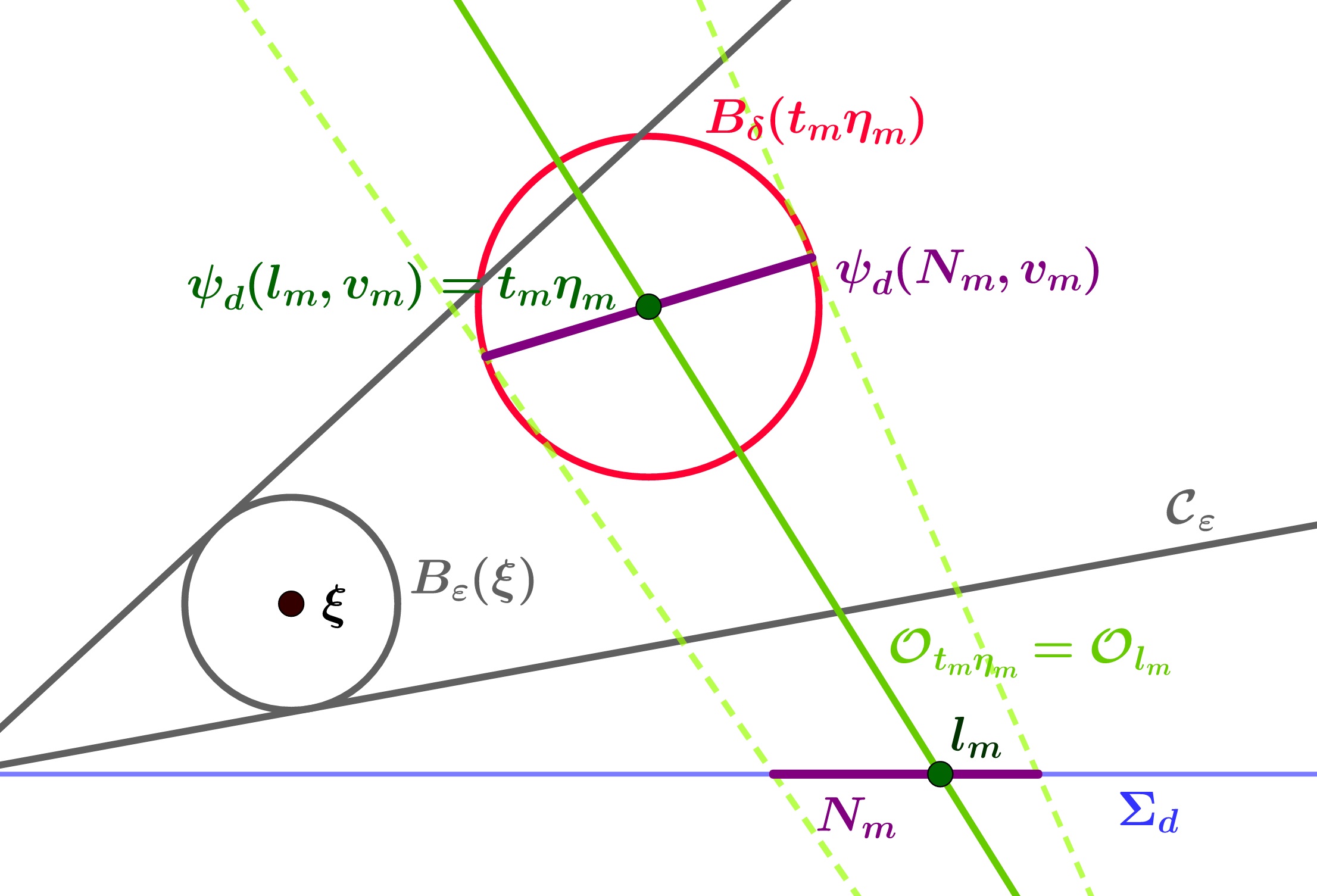}
\caption{The choice of $l_m$ and $N_m$.}
\label{fig:Nm}
\end{figure}

Applying the above Proposition \ref{prop:ReMcoeffsgeqSimpl} to  $\zeta_l$, $l\in N_m$, we obtain measurable, normalized vectors $u_{\zeta_l},v_{\zeta_l}\in\cH_{\zeta_l}$. Since $\sigma_l\cong\sigma_{\zeta_l}$ and $\cH_l\cong \cH_{\zeta_l}$ we have corresponding measurable, normalized vectors $u_{l},v_{l}\in\cH_{l}$. With these we define
\begin{align*}
u^{(m)}:= (\mu_\pi(N_m))^{-\frac 12} \int_{\Sigma_d} \mathds 1_{N_m}(l)u_l~d\mu_\pi(l) \; \in \cH_\pi,
\end{align*}
since the $u_l$ are measurable in $l$ and $\|u^{(m)}\|^2_{\cH_\pi}= (\mu_\pi(N_m))^{-1} \int_{\Sigma_d} \mathds 1_{N_m}(l)\|u_l\|^2~d\mu_\pi(l)=1$. We define $v^{(m)}\in\cH_\pi$ analogously. \\
Recall that Proposition \ref{prop:ReMcoeffsgeq} only gives us a lower bound for large $\|\zeta\|$ for functions $\phi$ with a small support, more precisely the support of $\phi$ shrinks proportional to $\langle\zeta\rangle^{-1+\varepsilon}$.
Thus, let $\phi\in C_c^\infty(B_{\frac{1}{C}}(0))$ be non-negative and $\varphi=\phi\circ\log$.
To adapt its support we define $\phi_m(X)\coloneq\langle t_m\rangle^{(1-\varepsilon)n}\phi(\langle t_m\rangle^{1-\varepsilon}X)\in C_c^\infty(B_{\frac 1C \langle\zeta\rangle^{-1+\varepsilon}}(0))$, $\varphi_m=\phi_m\circ\log$. With this choice $\|\phi_m\|_{L^1}=\|\phi\|_{L^1}$ and $\|\phi_m\|_{W^{K,1}}\leq \langle t_m\rangle^{K(1-\varepsilon)} \|\phi\|_{W^{K,1}}$
Then, by definition of $N_m$:
\begin{align}
|\cF(\langle & \pi(\bullet)u^{(m)},v^{(m)}\rangle\varphi_m)(t_m\eta_m)| \nonumber\\
		&= \left|\int_G \int_{N_m} (\mu_\pi(N_m))^{-1}  \langle\sigma_l(g)u_l,v_l\rangle \varphi_m(g)e^{-2\pi t_m\eta_m(\log g)}~dg~d\mu_\pi(l)  \right| \nonumber\\
		&\geq \left| \re\left(\int_G \int_{N_m} (\mu_\pi(N_m))^{-1}  \langle\sigma_l(g)u_l,v_l\rangle \varphi_m(g)e^{-2\pi t_m\eta_m(\log g)}~dg~d\mu_\pi(l)  \right)\right| \nonumber\\
		&= (\mu_\pi(N_m))^{-1} \left| \int_{N_m}\re\left(\int_\fg   \langle\sigma_l(\exp(X))u_l,v_l\rangle \phi_m(X)e^{-2\pi t_m\eta_m(X)}~dX\right)~d\mu_\pi(l)  \right| \nonumber\\
		&\overset{\text{Prop. }\ref{prop:ReMcoeffsgeqSimpl}}{\geq} (\mu_\pi(N_m))^{-1} \int_{N_m} \frac{1}{C} \langle \psi_d(l,v_m)\rangle^{-1/2} \|\phi_m\|_{L^1}~d\mu_\pi(l)\nonumber\\
		 &= \frac{1}{C}(\langle t_m\rangle -\delta)^{-1/2} \|\phi\|_{L^1}\|u^{(m)}\| \|v^{(m)}\|. \label{eq:final_estimate}
\end{align}
We can use this to show that $\xi\in \wf(\pi)$:
If  we assume that $\xi\notin \wf(\pi)$ we can employ  Lemma~\ref{lem:HoweKritsPhi} (see also (\ref{eq:HKritM})). It states that  there exist $\varepsilon_1,\varepsilon_2 >0$ such that for all $\varphi\in C_c^\infty(\exp(B_{\varepsilon_2}(0)))$ and all $N>n$:
\begin{align*}
|\cF(\langle & \pi(\bullet)u,v\rangle\varphi)(t\eta)| \leq C_{N}\|\varphi\|_{W^{N+n,1}} \|u\|\|v\| t^{-N} \quad \forall\; u,v\in\cH_\pi, \; \eta\in B_{\varepsilon_1}(\xi), \; t> t_0.
\end{align*}
For our sequence chosen above this means we would have 
\begin{align*}
|\cF(\langle \pi(\bullet)u^{(m)}&,v^{(m)}\rangle\varphi_m)(t_m\eta_m)| 
\leq C_{N}\|\varphi_m\|_{W^{N+n,1}} \|u^{(m)}\|\|v^{(m)}\| t_m^{-N} \\
&\leq C_{N}\|\varphi\|_{W^{N+n,1}} \|u^{(m)}\|\|v^{(m)}\| \langle t_m\rangle^{(N+n)(1-\varepsilon)} t_m^{-N}.
\end{align*}
Since $\langle t_m\rangle^{(N+n)(1-\varepsilon)} t_m^{-N}\in\cO\left(t_m^{n-\varepsilon N}\right)$ our estimations above show that this contradicts \eqref{eq:final_estimate} if $N$ is chosen big enough such that $n-N\varepsilon < -1/2$.

\end{proof}

\subsection{Proof of the Inclusion \texorpdfstring{$\wf(\pi)\subset \ac(\cO-\supp(\pi))$}{WF subset AC}}\label{subsec:2ndIncl}

For the proof of this inclusion we again find explicit microlocal estimates of individual matrix coefficients which we again obtain via induction over the dimension of $\fg$. For the formulation we use the notation introduced in Definition \ref{def:gprec} once more.
\begin{proposition}\label{prop:detaO}
Let $N$ be a nilpotent, connected, simply connected Lie group with  Lie algebra $\fn$ and fix an inner product on $\fn$ such that \eqref{eq:kappa_X_0} is defined. Then for any  $n,N\in\NN$ with $N>n$ there exists a constant $C_{n,N}>0$ such that for all nilpotent, connected, simply connected Lie groups $G$ with Lie algebra $(\fg,\langle ,\rangle_\fg)\prec(\fn,\langle ,\rangle_{\fn})$, $\dim\fg=n$, and  for all $1>\varepsilon >0$ there exists a neighborhood $U\subset \fg$ of $0$ such that the following estimate holds for all $\phi\in C_c^\infty(U)$, $l,\eta\in i\fg^\ast$ and all $u,v\in\cH_l$:
\begin{align*}
	\left| \int_{\fg} \langle \sigma_l(\exp(X))u,v\rangle_{\cH_l} \phi(X) e^{-2\pi\eta(X)}~dX  \right| 
	\leq C_{n,N}  \|u\|_{\cH_l} \|v\|_{\cH_l} \|\phi\|_{W^{N+n,1}(\fg)} \langle d(B_{\varepsilon\|\eta\|}(\eta),\cO_l)\rangle^{-N},
\end{align*}
where $dX$ the measure associated to the inner product on $\fg$.
\end{proposition}

For the proof of Proposition~\ref{prop:detaO} we distinguish the same two cases as in the proof of Proposition \ref{prop:ReMcoeffsgeq}. We want to outline our approach in each case:
\begin{itemize}
\item[i)] If $l(Z)=0$ for an $Z\in\fz(\fg)$, we consider $\overline{\fg}=\fg/(\RR\cdot Z)$, $\overline{l}=\operatorname{pr}_{i\overline{\fg}^\ast}(l)$ and find that $\sigma_l|_{\overline{G}}\cong \sigma_{\overline{l}}$ analogously to Case I of Theorem~\ref{thm:sigmaOg0}. Thus, we can express the Fourier transform of the matrix coefficient of $\sigma_l$ in terms of the Fourier transform of the corresponding matrix coefficient of $\sigma_{\overline{l}}$ and apply the induction hypothesis. 
To find the desired estimate we use the orbit structure $\operatorname{pr}_{i\overline{\fg}^\ast}(\cO_l)=\cO_{\overline{l}}$.

\item[ii)] If $\fz(\fg)=\RR\cdot Z$ and $l(Z)\neq 0$, Kirillov's Lemma \ref{lem:KirillovL} gives us a subalgebra $\fg_0$, $l_0=\operatorname{pr}_{i\fg_0^\ast}(l)$. Since we are in Case II of Theorem \ref{thm:sigmaOg0} we know that $\sigma_l\cong\operatorname{Ind}_{G_0}^G(\sigma_{l_0})$. Thus, we can express the Fourier transform of the matrix coefficient of $\sigma_l$ using the Fourier transform of the corresponding matrix coefficient of $\sigma_{l_0}$, apply the induction hypothesis and use the orbit picture $\operatorname{pr}_{i\fg_0^\ast}(\cO_\zeta)=\bigsqcup_{t\in\RR} (\Ad^\ast\exp tX)\cO_{\zeta_0}$ and $\cO_\zeta=\operatorname{pr}_{i\fg_0^\ast}^{-1}(\operatorname{pr}_{i\fg_0^\ast}(\cO_\zeta))$ in the estimates.
However, we again face some difficulties: 
In order to express the Fourier transform of the matrix coefficient of $\sigma_l$ using the Fourier transform of the corresponding matrix coefficient of $\sigma_{l_0}$ we use a chart $\fg\to G$ resulting from the decomposition $\fg=\fg_0\oplus \RR X$ given by the Kirillov Lemma. In order to switch to the desired chart $\exp:\fg\to G$ we
apply the Fourier inversion formula and use non-stationary phase. Due to the latter we have to consider neighborhoods whose radius grows proportional to the norm of its center. But this is no problem for us and actually matches the conical property of the wave front set and the asymptotic cone.
\end{itemize}

\begin{proof}
We prove this statement by induction on $\dim \fg$. If $n=\dim \fg=1$ or 2, the group is abelian. In this case the irreducible unitary representations are one-dimensional, $\sigma_l(g)=e^{2\pi l(\log g)}$, and have a zero-dimensional orbit $\cO_l=\{l\}$. 
We compute
\begin{align*}
	\left| \int_\fg \langle \sigma_l(\exp X)u,v\rangle_{\CC} \phi(X) e^{-2\pi\eta(X)}~dX  \right| 
	= \left| \int_\fg  \phi(g) e^{2\pi (l-\eta)(X)}u\overline{v}~dX  \right| 
	=  |\hat{\phi}(\eta-l)|\cdot |u|\cdot |v|.
\end{align*}
Fixing an inner product on $\fg$ we obtain a corresponding one on $i\fg^\ast$. Now let $\{X_i\}_{i=1}^n$ be an orthogonal basis for $\fg$ and pick $j\in\{1,\ldots, n\}$ such that $|(l-\eta)(X_j)|$ is maximal. \\
With this choice we have for $N\in\NN$ and $l\neq \eta$
\begin{align*}
 |\hat{\phi}(\eta-l)| &= \left| (2\pi (l-\eta)(X_j))^{-N} \int_\fg \phi(X)\partial_{X_j}^N e^{2\pi (l-\eta)(X)}dX \right|  \\
 &\leq (2\pi)^{-N} |(l-\eta)(X_j))|^{-N} \int_\fg |\partial_{X_j}^N\phi(X)|dX \\
 &\leq (2\pi)^{-N} \sqrt{n}^N \|l-\eta\|^{-N} \|\phi\|_{W^{N,1}(\fg)} \quad \leq C_{n,N}\langle d(\eta,l) \rangle^{-N} \|\phi\|_{W^{N+n,1}(\fg)}.
\end{align*}
The claim now follows with $U=\fg$ since $d(l,\eta)\geq d(B_{\varepsilon\|\eta\|}(\eta),l)$ for all $\varepsilon>0$.

Now we assume $n=\dim \fg\geq 3$. We will distinguish between the two cases: 

\textbf{Case I: $\boldsymbol{l(Z)=0}$ for an $\boldsymbol{Z\in\fz(\fg)}$.}
Given the inner product on $\fg$ let $W<\fg$ be the subspace such that $\fg=W\oplus\RR Z$ is an orthogonal decomposition. Then $\overline{\fg}=\fg/(\RR\cdot Z)$ is isomorphic to $W$ and has a well-defined Lie algebra structure $\lbrack v+\RR Z, w+\RR Z\rbrack =\lbrack v,w\rbrack_\fg +\RR Z$ since $Z\in\fz(\fg)$. \\
Given an inner product on $\overline{\fg}$ we choose one on $\fg$ such that the decomposition above is orthogonal. Furthermore, without loss of generality we may assume $\|Z\|=1$.
Using the corresponding inner product on $i\fg^\ast$ we also obtain an orthogonal decomposition $i\fg^\ast=iW^\ast\oplus i\RR Z^* \cong i\overline{\fg}^\ast \oplus i \RR Z^*$ with $\|Z^*\|=1$.\\
Note that $i\overline{\fg}^\ast$ is $\Ad^\ast(G)$-invariant (again due to  $Z\in\fz(\fg)$). We can identify $l$ and its orbit $\cO_l^G\subset i\fg^\ast$ with an element $\overline{l}\in i\overline{\fg}^\ast$ and its orbit $\cO_{\overline{l}}^{\overline{G}}\subset i\overline{\fg}^\ast$, respectively.\\
Let $\eta=\overline{\eta}+ir Z^*\in i\fg^\ast=i\overline{\fg}^\ast\oplus i\RR Z^*$. Then by the choice of the inner product we know $d(\eta,\cO_l^G)^2 = d(\overline{\eta},\cO_{\overline{l}}^{\overline{G}})^2+r^2$ and assuming $d(B_{\varepsilon\|\eta\|}(\eta),\cO_l)>0$ we can estimate
\begin{align*}
d(B_{\varepsilon\|\eta\|}(\eta),\cO_l)
&=d(\eta,\cO_l)-\varepsilon\|\eta\|
=\sqrt{d(\overline{\eta},\cO_{\overline{l}})^2+r^2}-\varepsilon\|\eta\|
\leq d(\overline{\eta},\cO_{\overline{l}})+r-\varepsilon\|\eta\|  \\
& \leq d(B_{\varepsilon\|\overline{\eta}\|}(\overline{\eta}),\cO_{\overline{l}})+\varepsilon\|\overline{\eta}\|+r-\varepsilon\|\eta\| 
\leq d(B_{\varepsilon\|\overline{\eta}\|}(\overline{\eta}),\cO_{\overline{l}}) +r,
\end{align*}
since $\|\overline{\eta}\|-\|\eta\|\leq 0$.
This implies that we are either in the case 
\begin{align}
	\text{a) } r \geq \frac 12 d(B_{\varepsilon\|\eta\|}(\eta),\cO_l^G) \quad \text{ or } \quad 
	\text{b) } d(B_{\varepsilon\|\overline{\eta}\|}(\overline{\eta}),\cO_{\overline{l}}^{\overline{G}}) \geq  \frac 12 d(B_{\varepsilon\|\eta\|}(\eta),\cO_l^G). \label{eq:casesabd}
\end{align}
Turning to the integral we want to estimate:
 \begin{align*}
 	J & :=\left| \int_{\fg} \langle \sigma_l(\exp(X))u,v\rangle_{\cH_l} \phi(X) e^{-2\pi\eta(X)}~dX  \right|  \\
 	& =\left|\int_{\overline{\fg}} \int_\RR \langle\sigma_l(\exp(\overline{X}+tZ))u,v\rangle_{\cH_l} \phi(\overline{X}+tZ)e^{-2\pi\eta(\overline{X}+tZ)}d\overline{X}~dt\right| \\
 	& =\left|\int_{\overline{\fg}} \int_\RR \langle\sigma_l(\exp(\overline{X})\exp(tZ))u,v\rangle_{\cH_l} \phi(\overline{X}+tZ)e^{-2\pi(\overline{\eta}(\overline{X})+ir Z^*(tZ))}d\overline{X}~dt\right|  \\
 	& = \left|\int_{\overline{\fg}} \int_\RR \langle\sigma_l(\exp(\overline{X}))u,v\rangle_{\cH_l} \phi(\overline{X}+tZ)e^{-2\pi(\overline{\eta}(\overline{X})+irt)}d\overline{X}~dt\right|
 \end{align*}
The last equality is due to $l(Z)=0$ which implies $\sigma_l(g\exp(tZ))=\sigma_l(g)$ for all $g\in G$, $t\in\RR$.

We start with case a) of (\ref{eq:casesabd}) and define
\begin{align*}
	\tilde{\phi}(t):= \int_{\overline{\fg}} \langle\sigma_l(\exp(\overline{X}))u,v\rangle_{\cH_l} \phi(\overline{X}+tZ)e^{-2\pi\overline{\eta}(\overline{X})}~d\overline{X} \; \in C_c^\infty(\RR).
\end{align*}
Then by integration by parts (as in the abelian case with $l=0$ and $u=v=1$) we obtain
\begin{align*}
	J  = \left| \int_\RR  \tilde{\phi}(t)e^{-2\pi rt}~dt\right|
	\leq C \|\tilde{\phi}\|_{W^{N,1}(\RR)} r^{-N} 
	\overset{(\ref{eq:casesabd}) a)}{\leq}  C_{N}  \|\tilde{\phi}\|_{W^{N+n,1}(\RR)} \langle d(B_{\varepsilon\|\eta\|}(\eta),\cO_l^G) \rangle^{-N}.
\end{align*}
The claim now follows in this case with the following estimation:
\begin{align*}
	\|\tilde{\phi}\|_{W^{N+n,1}(\RR)} & = \sum_{k=1}^{N+n} \|\partial_t^k \tilde{\phi}\|_{L^1(\RR,dt)} \\
	& \leq \sum_{k=1}^{N+n} \int_\RR \int_{\overline{\fg}} \left| \langle	\sigma_l(\exp(\overline{X}))u,v\rangle_{\cH_l} \partial_t^k\phi(\overline{X}+tZ)e^{-2\pi\overline{\eta}(\overline{X})}\right|~d\overline{X}~dt \\
	 \leq\quad \|u\| \|v\|  \|\phi\|_{W^{N+n,1}(\fg)}.
\end{align*}
Now let's turn to case b) of  (\ref{eq:casesabd}). Note that by Theorem \ref{thm:sigmaOg0} (i) we know $\cH_l\cong\cH_{\overline{l}}$ and $\sigma_{\overline{l}}\circ P\cong \sigma_l$  with the projection $P:G\to \overline{G}$. \\
Thus, we have
\begin{align*}
J &= \left|\int_{\overline{\fg}} \int_\RR \langle\sigma_{\overline{l}}(\exp(\overline{X}))u,v\rangle_{\cH_{\overline{l}}} \phi(\overline{X}+tZ)e^{-2\pi(\overline{\eta}(\overline{X})+irt)}d\overline{X}~dt\right|.
\end{align*}
Now define
\begin{align*}
	\check{\phi}(\overline{X}):= \int_\RR \phi(\overline{X}+tZ)e^{-2\pi irt}~dt \; \in C_c^\infty(\overline{\fg}),
\end{align*}
and choose the neighborhood $0\in U\subset \fg$ such that $\supp\check{\phi}\subset \overline{U}\subset \overline{\fg}$ given by the induction hypothesis applied to $\overline{\fg}$.
Then 
\begin{align*}
	J &= \left|\int_{\overline{\fg}}  \langle\sigma_{\overline{l}}(\exp(\overline{X}))u,v\rangle_{\cH_l} \check{\phi}(\overline{X})e^{-2\pi\overline{\eta}(\overline{X})}d\overline{X} \right| \\
	& \overset{\text{(IH)}}{\leq} C_{n-1,N} \|u\| \|v\| \|\check{\phi}\|_{W^{N+n-1,1}(\overline{\fg})} \langle d(B_{\varepsilon\|\overline{\eta}\|}(\overline{\eta}),\cO_{\overline{l}}^{\overline{G}}) \rangle^{-N} \\
	& \overset{(\ref{eq:casesabd}) b)}{\leq} C_{n,N} \|u\| \|v\| \|\check{\phi}\|_{W^{N+n,1}(\overline{\fg})} \langle d(B_{\varepsilon\|\eta\|}(\eta),\cO_l^G) \rangle^{-N}.
\end{align*}
The claim now follows in this case with the following estimation:
\begin{align*}
	\|\check{\phi}\|_{W^{N+n,1}(\overline{\fg})} & = \sum_{|\alpha|<N+n} \|\partial^\alpha\check{\phi}\|_{L^1(\overline{\fg},dv)} \\
	& = \sum_\alpha \int_{\overline{\fg}} \left| \int_\RR \partial_{\overline{X}}^\alpha\phi(\overline{X}+tZ)e^{-2\pi irt}~dt \right|~d\overline{X} \\
	& \leq \sum_\alpha \int_{\overline{\fg}} \int_\RR \left| \partial_{\overline{X}}^\alpha\phi(\overline{X}+tZ)\right|~dt~d\overline{X} \quad \leq \quad  \|\phi\|_{W^{N+n,1}(\fg)}.
\end{align*}

\textbf{Case II: $\boldsymbol{\fz(\fg)=\RR\cdot Z}$ and $\boldsymbol{l(Z)\neq 0}$.} 
Kirillov's Lemma \ref{lem:KirillovL} gives us $X,Y\in\fg$ and an ideal $\fg_0\subset\fg$ with $\fg=\RR X \oplus \fg_0$ and $[X,Y]=Z$. We may choose $X$ such that this decomposition is orthogonal and $\|X\| = 1$.
Since $\dim(\fz(\fg_0))>1$ as $Z,Y\in\fz(\fg_0)$
we are in Case I in the induction hypotheses for $G_0$
. We define a chart for $G$ via 
 $$\beta:\fg=\fg_0 \oplus\RR X \to G, \quad X_0+tX\mapsto\exp(X_0)\exp(tX).$$
Since $X\notin\fr_l$ and we are in Case II of Proposition \ref{prop:cases} and Theorem \ref{thm:sigmaOg0}:
\begin{align*}
	& p:i\fg^\ast\to i\fg^\ast_0, \quad l_0:=p(l), \eta_0:=p(\eta),\;  \cO_{l_0}^{G_0}:=\Ad^\ast(G_0)l_0, \\
	& p(\cO_l^G)=\bigsqcup_{t\in\RR} (\Ad^\ast\exp tX)\cO_{l_0}^{G_0},\quad \cO_l^G=p^{-1}(p(\cO_l^G)).
\end{align*}
where $G_0=\exp(\fg_0)\subset G$ is a normal subgroup.
Note that we also have an orthogonal decomposition $i\fg^\ast=i\RR X^* \oplus i\fg^\ast_0$, $X^*(X)=1$, which gives us for all $a\in A=\exp(\RR X)$:
\begin{align*}
	d(\eta_0,\cO_{\Ad^\ast(a)l_0}^{G_0})=d(\eta_0,\Ad^\ast(a)\cO_{l_0}^{G_0})\geq d(\eta_0,p(\cO_l^G)) = d(\eta, \cO_l^G).
\end{align*}
Assuming $d(B_{\varepsilon\|\eta\|}(\eta),\cO_l^G)>0$ we can estimate 
\begin{align}\label{eq:d0}
d(B_{\varepsilon\|\eta\|}(\eta),\cO_l^G) &= d(\eta,\cO_l^G) -\varepsilon\|\eta\|
\leq d(\eta_0,\cO_{\Ad^\ast(a)l_0}^{G_0})-\varepsilon\|\eta\|  \\
&= d(B_{\varepsilon\|\eta_0\|}(\eta_0),\cO_{\Ad^\ast(a)l_0}^{G_0})+\varepsilon\|\eta_0\|-\varepsilon\|\eta\| \leq d(B_{\varepsilon\|\eta_0\|}(\eta_0),\cO_{\Ad^\ast(a)l_0}^{G_0}), \nonumber
\end{align}
since $\|\eta_0\|-\|\eta\|\leq 0$.
In addition to that we have $\eta=\eta_0+irX^*$ with $iX^* \in\ker(p)$.

We start by estimating the following integral and deal with the transition from the chart $\beta$ to the exponential chart later on.
\begin{align*}
	J(\phi,\eta) &:=\left| \int_{\fg} \langle \sigma_l(\beta(X))u,v\rangle_{\cH_l} \phi(X) e^{-2\pi\eta(X)}~dX  \right|  \\
 	& =\left|\int_{\fg_0} \int_\RR \langle\sigma_l(\exp(X_0)\exp(tX))u,v\rangle_{\cH_l} \phi(X_0+tX)e^{-2\pi(\eta_0(X_0)+ irX^*(tX))}dX_0~dt\right|.
\end{align*}
By Theorem \ref{thm:sigmaOg0}, we also know $\sigma_l\cong \operatorname{Ind}_{G_0}^G(\sigma_{l_0})$. Note that $\cH_l\cong  L^2(A,\cH_{l_0})$. If we regard $u$ and $v$ as elements of $ L^2(A,\cH_{l_0})$ and $\tilde{u},\tilde{v}:G\to\cH_{l_0}$ the corresponding functions in the 'standard model' we have again
\begin{align*}
	& \langle \sigma_l(g_0 a)u,v\rangle_{\cH_l} = \int_A \langle \lbrack\sigma_l(g_0 a)u\rbrack (b), v(b) \rangle_{\cH_{l_0}}~db \quad \text{ and } \\
	& \lbrack\sigma_l(g_0a)\tilde{u}\rbrack (b) = \tilde{u}(bg_0 a)=\tilde{u}(bg_0b^{-1}b a)=\sigma_{l_0}(bg_0b^{-1})\tilde{u}(ba) 
\end{align*}
since $b^{-1}g_0b\in G_0$ as $\fg_0$ is an ideal. This gives us $\lbrack\sigma_l(g_0a)u\rbrack (b) = \sigma_{l_0}(bg_0b^{-1})u(ba)$.\\
We deduce that
\begin{align*}
		J(\phi,\eta) &= \bigg|\int_{\fg_0} \int_\RR \left(\int_A \langle\sigma_{l_0}(b\exp(X_0)b^{-1})u(be^{tX}),v(b)\rangle_{\cH_{l_0}}db\right)\cdot \\
		& \qquad\qquad\qquad\qquad\qquad\qquad\qquad\qquad\qquad\quad \phi(X_0+tX)e^{-2\pi(\eta_0(X_0)+ir X^*(tX))}dX_0~dt\bigg| \\
		& \leq \int_\RR \int_A \left|\int_{\fg_0} \langle\sigma_{l_0}(b\exp(X_0)b^{-1})u(be^{tX}),v(b)\rangle_{\cH_{l_0}} \phi(X_0+tX)e^{-2\pi\eta_0(X_0)} dX_0\right|\cdot \\
		& \qquad\qquad\qquad\qquad\qquad\qquad\qquad\qquad\qquad\qquad\qquad\qquad\qquad\qquad \left| e^{-2\pi i rt} \right| db~dt.
\end{align*}
The conjugation $C_b:G_0\to G_0, g_0\mapsto b^{-1}g_0b$ is a group automorphism and we know that $\chi_{l_0}\circ C_b = \chi_{\Ad^\ast(b)l_0}$ for the character $\chi_{l_0}$ such that $\sigma_{l_0}=\Ind_M^{G_0}(\chi_{l_0})$, $M=\exp(\fm)$ for a polarizing subalgebra $\fm\subset \fg_0$.
Now, $\Ad(b)\fm$ is a polarizing subalgebra for $\Ad^\ast(b)l_0$ and $C_b^{-1}M = \exp(\Ad(b)\fm)$.
Thus, \cite[Lemma 2.1.3]{corgre} gives us
\begin{align*}
\sigma_{\Ad^\ast(b)l_0}=\Ind_{C_b^{-1}M}^{G_0}(\chi_{l_0}\circ C_b) \cong \Ind_M^{G_0}(\chi_{l_0})\circ C_b = \sigma_{l_0}\circ C_b.
\end{align*}
We choose $\tilde U \subset \fg$ such that for all $\phi\in C_c^\infty(\tilde U)$ and $X_0+tX\in \tilde U$ we have $\supp(\phi(\bullet+tX))\subset U_0\subset\fg_0$, where $0\in U_0\subset\fg_0$ is given by the induction hypothesis for $G_0$. We apply the induction hypothesis for $G_0$ on the decay of the Fourier transform to $\sigma_{\Ad^*(b^{-1})l_0}$:
\begin{align*}
		J&(\phi,\eta) \leq \int_\RR \int_A \left|\int_{\fg_0} \langle\sigma_{l_0}(b\exp(X_0)b^{-1})u(be^{tX}),v(b)\rangle_{\cH_{l_0}} \phi(X_0+tX)e^{-2\pi\eta_0(X_0)} dX_0\right| db~dt \\
		& \overset{\text{(IH)}}{\leq} \int_\RR \int_A C_{n-1,N} \|\phi(\bullet+tX)\|_{W^{N+n-1,1}(\fg_0)}\|u(be^{tX})\|_{\cH_{l_0}} \|v(b)\|_{\cH_{l_0}} \cdot\\
		& \qquad\qquad\qquad\qquad\qquad\qquad\qquad\qquad\qquad\qquad\qquad\qquad \langle d(B_{\varepsilon\|\eta_0\|}(\eta_0),\cO_{\Ad^\ast(b^{-1})l_0}^{G_0})\rangle^{-N} ~db~dt  \\
		&\overset{(\ref{eq:d0})}{\leq} C_{n-1,N} \langle d(B_{\varepsilon\|\eta\|}(\eta),\cO_{l})\rangle^{-N} \int_\RR \left(\int_A  \|T_{\exp(tX)}u(b)\|_{\cH_{l_0}} \|v(b)\|_{\cH_{l_0}}~db\right)\|\phi(\bullet+tX)\|_{W^{N+n-1,1}}~dt  \\
		&\leq\; C_{n-1,N} \langle d(B_{\varepsilon\|\eta\|}(\eta),\cO_{l})\rangle^{-N} \int_\RR \|T_{\exp(tX)}u\|_{\cH_l} \|v\|_{\cH_l}\|\phi(\bullet+tX)\|_{W^{N+n,1}(\fg_0)}~dt,
	\end{align*}
	where $T_{\exp(tX)}$ is the translation by $\exp(tX)\in A$ which is an isometry on $\cH_l\cong L^2(A,\cH_{l_0})$. This gives us
	\begin{align*}
		J(\phi,\eta) & \leq\; C_{n,N} \langle d(B_{\varepsilon\|\eta\|}(\eta),\cO_{l})\rangle^{-N} \|u\|_{\cH_l}\|v\|_{\cH_l} \int_\RR \|\phi(\bullet+tX)\|_{W^{N+n,1}(\fg_0)}~dt \\
		& =\; C_{n,N} \langle d(B_{\varepsilon\|\eta\|}(\eta),\cO_{l})\rangle^{-N} \|u\|_{\cH_l}\|v\|_{\cH_l} \int_\RR \sum_{|\alpha|\leq N+n} \int_{\fg_0} |\partial_{X_0}^N \phi(X_0+tX)|dX_0~dt\\
		& \leq \; C_{n,N} \langle d(B_{\varepsilon\|\eta\|}(\eta),\cO_{l})\rangle^{-N} \|u\|_{\cH_l}\|v\|_{\cH_l} \|\phi\|_{W^{N+n,1}(\fg)}.
\end{align*}
Now let $\kappa=\beta^{-1}\circ \exp:B_1(0)\to\fg$ be the transition map. Then the integral we are interested in can be written as
\begin{align*}
F&:= \left| \int_{\fg} \langle \sigma_l(\exp(X))u,v\rangle_{\cH_l} \phi(X) e^{-2\pi\eta(X)}~dX  \right| \\
&= \left| \int_{\fg} \langle \sigma_l(\beta(\kappa(X)))u,v\rangle_{\cH_l} \chi(\kappa(X))\phi(X) e^{-2\pi\eta(X)}~dX  \right|,
\end{align*}
where $\chi\in C_c^\infty(\fg)$ is a cut-off function with $\chi=1$ on $\kappa(\supp(\phi))$.
The Fourier inversion formula yields
\begin{align*}
F &= \left| \int_{i\fg^\ast} \cF\left(\langle \sigma_l(\beta(\bullet))u,v\rangle_{\cH_l} \chi(\bullet)\right)(\xi)\int_{\fg}\phi(X)e^{-2\pi(\eta(X)-\xi(\kappa(X))}~dX ~d\xi\right|.
\end{align*}
Now, $\cF\left(\sigma_l(\beta(\bullet))u,v\rangle_{\cH_l} \chi(\bullet)\right)(\xi)= J(\chi,\xi)$ from above. We later want to apply the estimates for $J$ from above to $J(\chi, \xi)$ and thus require $\supp(\chi)\subset \tilde U$, where $\tilde U\subset \fg$ has be defined above. In order to make this possible we choose $U\subset \fg$ such that
\begin{itemize}
 \item $U \subset B_1(0)$ (this assures that $\kappa$ is well defined)
 \item $\kappa(\overline U) \subset \tilde U$
\end{itemize}

Furthermore, we can use non-stationary phase to estimate the inner integral
\begin{align*}
I(\phi, \xi,\eta):=\int_{\fg}\phi(X)e^{-2\pi(\eta(X)-\xi(\kappa(X))}~dX = \int_{\fg}\phi(X)e^{-2\pi d_\varepsilon(\eta(X)-\xi(\kappa(X))/d_\varepsilon}~dX,
\end{align*}
where $d_\varepsilon=d(B_{\varepsilon\|\xi\|}(\xi),\eta)>0$ is assumed.
With the phase function $f_{\xi,\eta}(X):=\frac 1{d_\varepsilon} (\eta(X)-\xi(\kappa(X))$ we have 
$d_X f_{\xi,\eta}(X):=\frac 1{d_\varepsilon} (\eta-\xi\circ D\kappa(X)))$ where $D\kappa(X):\fg\to\fg$ is the differential of $\kappa$ in $X$. Since $D\kappa(0)=1$ we have  
\begin{align*}
\|\xi\circ D\kappa(X)-\xi\|\leq \sup_{X\in U} \|D\kappa(X)-1\|\|\xi\|\leq \varepsilon\|\xi\|,
\end{align*}
after possibly shrinking the neighborhood $0\in U\subset \fg$. 
This gives us
\begin{align*}
\|\eta-\xi\circ D\kappa(X)\|\geq \|\eta-\xi\|-\varepsilon\|\xi\|=d_\varepsilon \quad \Rightarrow \quad
\left| d_X f_{\xi,\eta}(X) \right| \geq 1.
\end{align*}
With \cite[Theorem 7.7.1]{hoerm} we can estimate
\begin{align*}
|I(\phi, \xi,\eta)|\leq C_N\|\kappa\|_{C^{N+1}(B_1(0))} \|\phi\|_{W^{N,1}(\fg)} d(B_{\varepsilon\|\xi\|}(\xi),\eta)^{-N}.
\end{align*}
Note that Hörmander uses on the right hand side instead of the Sobolev norm of $\phi$ the term $\sum_{|\alpha|\leq N} \sup_X |D^\alpha\phi(X)|$. But when you take a closer look at his proof one finds that these suprema occur as an estimate of the integral of $\phi$. Hence, they can be replaced by the Sobolev norm.
Furthermore, by Lemma \ref{lem:Ckappa} we have $\|\kappa\|_{C^{N+1}}\leq C_{\fg,N}\leq C_{\fn,N}$ and therefore can be absorbed into the constant $C_N$ (since this may depend on $\fn$ in our statement).\\
In order to prove the desired estimate it suffices to prove it in the case that $d(B_{\varepsilon\|\eta\|}(\eta), \cO_l)>0$ which is equal to 
\begin{align}\label{eq:epseta<d}
\varepsilon\|\eta\|< d(\eta, \cO_l)
\end{align}
 and implies that $\frac 12 d:=\frac 12 d(\eta, \cO_l)<d(B_{\varepsilon/3\|\eta\|}(\eta), \cO_l)=:d_{\varepsilon/3}$.
 Now, we split up the integral:
\begin{align*}
F_I:= \int_{B_{1/2d_{\varepsilon/3}}(\eta)} J(\chi,\xi) \left|I(\phi,\xi,\eta)\right|~d\xi  , \quad
F_{II}:=  \int_{i\fg^\ast\setminus B_{d/4}(\eta)} J(\chi,\xi) \left| I(\phi,\xi,\eta)\right|~d\xi  .
\end{align*}
Since the two domains of integration are overlapping we have $F\leq F_I + F_{II}$.\\
With the estimates above (with $\frac \varepsilon 9$ instead of $\varepsilon$) we obtain
\begin{align*}
F_I &\leq C_{n,N}\sup_{\xi\in B_{1/2d_{\varepsilon/3}}(\eta)}\langle d(B_{\varepsilon/9\|\xi\|}(\xi),\cO_{l})\rangle^{-N} \|u\|_{\cH_l}\|v\|_{\cH_l} \|\phi\|_{W^{N,1}(\fg)} \|\chi\|_{W^{N+n,1}(\fg)} d_{\varepsilon/3}^n.
\end{align*}
For all $\xi\in B_R(\eta)$, $R=\frac 12 d_{\varepsilon/3}\leq \frac 12 d(\eta, \cO_l)$, we can estimate 
\begin{align}
d(B_{\varepsilon/9\|\xi\|}(\xi),\cO_{l}) &\geq d(\eta, \cO_l) -R -\frac{\varepsilon}9\|\xi\| \geq d(\eta, \cO_l)-(1+\varepsilon/9)R-\frac{\varepsilon}9\|\eta\| \\
&\geq \frac 13 (d(\eta, \cO_l)-\varepsilon/3\|\eta\|)=\frac 13 d(B_{\varepsilon/3\|\eta\|}(\eta), \cO_l). \nonumber
\end{align}
This gives us 
\begin{align*}
F_I &\leq C_{n,N}\langle d(B_{\varepsilon/3\|\eta\|}(\eta), \cO_l)\rangle^{-N+n} \|u\|_{\cH_l}\|v\|_{\cH_l} \|\phi\|_{W^{N,1}(\fg)} \\
&\leq C_{n,N}\langle d(B_{\varepsilon\|\eta\|}(\eta), \cO_l)\rangle^{-N+n} \|u\|_{\cH_l}\|v\|_{\cH_l} \|\phi\|_{W^{N,1}(\fg)},
\end{align*}
since $d(B_{\varepsilon\|\eta\|}(\eta), \cO_l)\leq d(B_{\varepsilon/3\|\eta\|}(\eta), \cO_l)$ and $-N+n<0$.\\
For the second part we use the above estimates again with $\frac \varepsilon 9$ instead of $\varepsilon$:
\begin{align*}
F_{II} &\leq C_N \|\phi\|_{W^{N,1}(\fg)} \|\chi\|_{W^{N,1}(\fg)}\|u\|_{\cH_l}\|v\|_{\cH_l}\int_{i\fg^\ast\setminus B(d/4,\eta)} d(B_{\varepsilon/9\|\xi\|}(\xi),\eta)^{-N}~d\xi.
\end{align*}
We estimate with $r=\|\xi-\eta\|\geq \frac 14 d(\eta,\cO_l)$ and $\varepsilon<1$
\begin{align*}
d(B_{\varepsilon/9\|\xi\|}(\xi),\eta) &= \|\xi-\eta\|-\frac\varepsilon 9\|\xi\| \geq \left(1-\frac\varepsilon 9\right)r-\frac\varepsilon 9 \|\eta\| \overset{\eqref{eq:epseta<d}}{\geq} \left(1-\frac\varepsilon 9\right)r - \frac 19 d(\eta,\cO_l) \\
&\geq \left( 1-\frac\varepsilon 9 -\frac 49 \right)r
\geq \frac 49 r
\end{align*}
and therefore with polar coordinates
\begin{align*}
F_{II} &\leq C_N \|\phi\|_{W^{N,1}(\fg)} \|u\|_{\cH_l}\|v\|_{\cH_l} \left(\frac 49\right)^{-N}\mathrm{vol(S^{n-1}})\int_{d/4}^\infty r^{-N+n-1}~dr \\
&=  C_N \|\phi\|_{W^{N,1}(\fg)} \|u\|_{\cH_l}\|v\|_{\cH_l} \left(\frac 49\right)^{-N}\mathrm{vol(S^{n-1}})\frac{1}{4^{N-n}} d(\eta,\cO_l)^{-N+n} \\
& \leq C_{n,N} \|\phi\|_{W^{N,1}(\fg)} \|u\|_{\cH_l}\|v\|_{\cH_l} d(B_{\varepsilon\|\eta\|}(\eta),\cO_l)^{-N+n},
\end{align*}
since $d(B_{\varepsilon\|\eta\|}(\eta),\cO_l)\leq d(\eta,\cO_l)$ and $-N+n<0$.
\end{proof}

\begin{corollary}\label{cor:deta0}
The statement of the previous Proposition \ref{prop:detaO} also holds for $u,v\in \cH_l^{\oplus m_l}$ with multiplicity $m_l\in\NN\cup\{\infty\}$.
\end{corollary}

\begin{proof}
For $u\in \cH_l^{\oplus m_l}$ we have $u=(u_1,u_2,\ldots)$ with (finitely or infinitely many) $0\neq u_i\in\cH_l$ and $\sum_i \|u_i\|^2_{\cH_l} < \infty$, $\|u\|=\left(\sum_i \|u_i\|^2\right)^{1/2}$. Thus
\begin{align*}
	\big| \int_{\fg} \langle \sigma_l(\exp(X))u,v\rangle_{\cH_l}& \phi(X) e^{-2\pi\eta(X)}~dX  \big|
	 = \left| \int_{\fg}\sum_i \langle \sigma_l(\exp(X))u_i,v_i\rangle_{\cH_l} \phi(X) e^{-2\pi\eta(X)}~dX  \right| \\
	& = \left| \sum_i \int_{\fg} \langle \sigma_l(\exp(X))u_i,v_i\rangle_{\cH_l} \phi(X) e^{-2\pi\eta(X)}~dX  \right|  \\ 
	& \overset{\text{Prop. }\ref{prop:detaO}}{\leq}  C_{n,N} \|\phi\|_{W^{N+n,1}(\fg)} \langle d(B_{\varepsilon\|\eta\|}(\eta),\cO_l)\rangle^{-N} \sum_i \|u_i\|\cdot\|v_i\| \\
	& \leq  C_{n,N} \|\phi\|_{W^{N+n,1}(\fg)} \langle d(B_{\varepsilon\|\eta\|}(\eta),\cO_l)\rangle^{-N} \left(\sum_i \|u_i\|^2\right)^{1/2}\cdot \left(\sum_i\|v_i\|^2\right)^{1/2} \\
	& =  C_{n,N} \|\phi\|_{W^{N+n,1}(\fg)} \langle d(B_{\varepsilon\|\eta\|}(\eta),\cO_l)\rangle^{-N} \|u\|\cdot\|v\|,
\end{align*}
where the interchanging of the order of integration and summation in the second equality is possible since $|\langle \sigma_l(\exp(X))u_i,v_i\rangle \phi(X) e^{-2\pi\eta(X)}|\leq \|u_i\|\cdot\|v_i\|\cdot |\phi(X)|\in L^1(\NN\times \fg)$.
\end{proof}

This inequality whose constant is in particular independent of $l\in i\fg^\ast$ now helps us to estimate the matrix coefficients of the big unitary representation $\pi$ using its direct integral decomposition into the irreducibles $\sigma_l$.

\begin{theorem}\label{thm:WfInAcGen}
Let $G$ be a nilpotent, connected, simply connected Lie group with Lie algebra $\fg$ and $(\pi,\cH_\pi)$ a unitary representation of $G$. Then
\begin{align*}
\wf(\pi)\subset \ac(\cO-\supp\pi).
\end{align*}
\end{theorem}
\begin{proof}
Let $\eta\notin\ac(\cO-\supp\pi)$, w.l.o.g. $\|\eta\|=1$. Then there exist $\varepsilon>0$ and $t_0>0$ such that $d(t\eta,\cO-\supp\pi)\geq 2\varepsilon t$ for all $t\geq t_0$. In particular,  for all $l\in\supp\pi$ we know $d(t\eta,\cO_l)\geq 2\varepsilon t$ which implies $d(B_{\varepsilon t}(t\eta),\cO_l)\geq \varepsilon t$.\\
Again, we use $\cH_\pi=\int_{\Sigma_d} \cH_l^{\oplus m(\pi,\sigma_l)}~d\mu_\pi(l)$ for the Hilbert space of the unitary representation $\pi$.
	If $u=(u_l), v=(v_l)\in\cH$, $u_l,v_l\in\cH_l^{\oplus m(\pi,\sigma_l)}$, in this direct integral decomposition the matrix coefficient is
	\begin{eqnarray*}
		\langle \pi(g)u,v\rangle = \int_{\Sigma_d} \langle \sigma_l(g)u_l,v_l\rangle~d\mu_\pi(l).
	\end{eqnarray*}	
Let $U\subset \fg$ be the neighborhood of $0$ from Proposition~\ref{prop:detaO}/Corollary \ref{cor:deta0} with $\varepsilon$ as chosen above. We want to use Definition~\ref{def:WFvs} in order to estimate the Wavefront set and thus choose $\phi\in C_c^\infty(U)$ with $\phi(0)\neq 0$. For $t\geq t_0$ and $\varphi:=\phi\circ\log \in C_c^\infty(G)$, $\varphi(e)\neq 0$, we conclude
	\begin{align*}
		|\cF(\langle \pi(\bullet) u,v\rangle\varphi)(t\eta)| &= \left| \int_G \langle \pi( g)u,v\rangle \varphi(g) e^{-2\pi t\eta(\log g)}~dg\right| \\
			&= \left| \int_G \int_{\Sigma_d} \langle \sigma_l(g)u_l,v_l\rangle \varphi( g) e^{-2\pi t\eta(\log g)}~d\mu_\pi(l)~dg \right| \\
			&= \left| \int_{\Sigma_d} \left(\int_G \langle \sigma_l( g)u_l,v_l\rangle \phi(\log g) e^{-2\pi t\eta(\log g)}~dg\right)~d\mu_\pi(l) \right| \\
			& \leq \int_{\Sigma_d} \left| \int_{\fg} \langle \sigma_l(\exp(X))u_l,v_l\rangle \phi(X) e^{-2\pi t\eta(X)}~dX \right| ~d\mu_\pi(l) \\
			& \overset{\text{Cor. }\ref{cor:deta0}}{\leq }\int_{\Sigma_d}  C_{n,N}  \|u_l\|_{\cH_l} \|v_l\|_{\cH_l} \|\phi\|_{W^{N+n,1}(\fg)} \langle d(B_{\varepsilon t}(t\eta),\cO_l)\rangle^{-N}~d\mu_\pi(l) \\
			& \leq C_{n,N}  \|\phi\|_{W^{N,1}(\fg)} \varepsilon^{-N}t^{-N} \int_{\Sigma_d}\|u_l\|\cdot\|v_l\|~d\mu_\pi(l) \\
			& \leq C_{n,N}  \|\phi\|_{W^{N,1}(\fg)}\varepsilon^{-N}\|u\|_{\cH_{\pi}}\cdot\|v\|_{\cH_{\pi}}t^{-N}.
	\end{align*}
This implies $\eta\notin\wf_e(\langle\pi(\bullet)u,v\rangle)$.
\end{proof}
\section{Example: Restriction of discrete series of $\mathrm{SU}(1,2)$}
In this final section we want to illustrate how the main theorem can be used to obtain results for the restriction of discrete series representations of real reductive groups to nilpotent subgroups.
In order to keep the notation simple, we restrict to the simplest non-trivial case, i.e. $G=\mathrm{SU}(1,2)$ and $N\subset G$ the 3 dimensional Heisenberg group that appears as the unipotent radical of the parabolic subgroup.

Let us fix some notation that is necessary to state the result: In a first step we have to fix an Iwasawa decomposition. We consider
$\mathrm{SU}(1,2)$ as the group of complex $3\times3$ matrices preserving the Hermitian form $|z_1|^2 - |z_2|^2 - |z_3|^2$. We fix the Cartan involution on $\mathfrak{su}(1,2)$ to be $\theta: X\mapsto -X^*$ and obtain a Cartan decomposition $\mathfrak{su}(1,2) = \mathfrak k + \mathfrak p$ with
\[
 \mathfrak k = \left\{ \left(\begin{array}{cc} -\tr(C)&0\\
                             0&C
                            \end{array}\right),~~ C\in\mathfrak u (2)
\right\} \textup{ and } \mathfrak p = \left\{ \matz{0}{b^*}{b}{0}, b\in \mathbb C^2\right\}\cong \mathbb R^4
\]
and introduce the Lie algebra elements
\[
 H= \matd{0}{0}{1}{0}{0}{0}{1}{0}{0}, X_{\pm \alpha, 1} =  \matd{0}{1}{0}{1}{0}{\mp 1}{0}{\pm 1}{0} , X_{\pm \alpha, 2} =  \matd{0}{i}{0}{-i}{0}{\pm i}{0}{\pm i}{0}, X_{\pm 2\alpha} = \matd{\mp i}{0}{i}{0}{0}{0}{-i}{0}{\pm i}
\]
With this notation we fix a maximal abelian subspace $\mathfrak a = \mathbb R \cdot H\subset \mathfrak p$ and obtain the simple real roots 
$\pm \alpha (H) = \pm 1$ with two dimensional root spaces $\mathfrak g_{\pm\alpha} = \mathbb R\cdot X_{\pm\alpha, 1} +   \mathbb R\cdot X_{\pm\alpha,2}$ 
as well as the root $\pm 2\alpha(H) = \pm 2$ with one dimensional root spaces $\mathfrak g_{\pm2\alpha} = \mathbb R \cdot X_{\pm2\alpha}$. 
We thus obtain the Iwasawa decomposition $\mathfrak{su}(1,2) = \mathfrak k\oplus\mathfrak a\oplus\mathfrak n$ with $\mathfrak n = \mathbb RX_{\alpha,1}+\mathbb RX_{\alpha,2}+\mathbb R X_{2\alpha}$ being the Lie algebra of the 3 dimensional Heisenberg group with center $\mathbb RX_{2\alpha}$.
Let $\mathrm{SU}(1,2) = KAN$ be the corresponding Iwasawa decomposition on the group level with $N\subset \mathrm{SU}(1,2)$ the Heisenberg group corresponding to $\mathfrak n \subset \mathfrak {su}(1,2)$.
The irreducible representations and orbits of $N$ are then easily described:
All points in $\{\xi\in \mathfrak n^*, \xi(X_{2\alpha}) =0\}$ are zero dimensional orbits corresponding to the one-dimensional representations of the Heisenberg group and all $\mathcal O_h = \{\xi\in \mathfrak n^*, \xi(X_{2\alpha}) =h \}, h\in \mathbb R\setminus\{0\}$ are coadjoint orbits that correspond to the Schrödinger representations $\rho_h$ (cf. \cite{corgre}).

Let us introduce the necessary notation in order to describe the three different types of discrete series for $SU(1,2)$, namely the holomorphic, anti holomorphic and non holomorphic discrete series. They are distinguished by their \emph{Harish Chandra parameter} which we want to introduce next. First note that $\mathfrak{su}(1,2)$ fulfills the rank condition $\mathrm{rk} (\mathfrak u(2)) = \mathrm{rk} (\mathfrak {su}(1,2)) =2$ thus we can choose $\mathfrak h \subset \mathfrak k$ which is a Cartan subalgebra in $\mathfrak k_{\mathbb C}$ and $\mathfrak g_{\mathbb C}$ at the same time. Let us denote
\[
 T_1 = \matd{-i}{0}{0}{0}{i}{0}{0}{0}{0} \textup{ and } T_2 = \matd{-i}{0}{0}{0}{0}{0}{0}{0}{i}
\]
and fix the nondegenerate $\mathrm{Ad}$-invariant scalar product on $\mathfrak{su}(1,2)$ to be $\langle X, Y\rangle := -\tr(XY)$ such that $\langle T_{1/2},T_{1/2}\rangle = 2$ and $\langle T_1, T_2\rangle = 1$, then one calculates that the roots on $\mathfrak g_{\mathbb C}$ are (modulo a factor of i) given by:
\[
 \pm \alpha_1 = \langle T_1,\bullet\rangle, \pm \alpha_2 = \langle T_2,\bullet\rangle \text{ and } \pm \alpha_3 = \langle T_2-T_1,\bullet\rangle.
\]
Note that $i\alpha_3$ is a compact root (i.e. a root also appearing of the $\mathfrak h$ action on $\mathfrak k$) and $\Delta_+ =\{\alpha_1, \alpha_2, \alpha_3\}$ is a choice of positive roots that is called \emph{good} in the sense that compact roots are never bigger then non-compact ones (cf. \cite[Chapter VI]{Kna86}). The discrete series of semisimple Lie groups are parametrized by their Harish Chandra parameter $\Lambda \in \mathfrak h^*$ and one has the following distinction between different types of discrete series: A discrete series is called \emph{holomorphic discrete series} if $\langle \Lambda, \alpha_i\rangle\geq 0$ for all non-compact roots, it is called \emph{anti-holomorphic} if $\langle \Lambda, \alpha_i\rangle\leq 0$  for all noncompact roots and \emph{non-holomorphic discrete series} else \cite[Chapter VI\&IX]{Kna86} see Figure~\ref{fig:roots} for a visualization).

\begin{figure}[ht] \centering
\includegraphics[width=.3\textwidth]{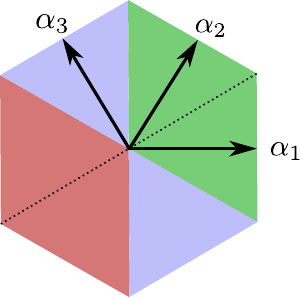}
\caption{Visualization of the Cartan subalgebra $\mathfrak h$: The compact root is given by $\alpha_3$ and the dashed line represents the reflection hyperplane of the compact Weyl group. The Harish Chandra parameter of the holomorphic discrete series lie in the green region, those of the anti-holomorphic discrete series in the red region and the non-holomorphic ones in the blue region.}
\label{fig:roots}
\end{figure}
We can now state the result on the restriction of discrete series.
\begin{proposition}
 Let $\pi$ be a discrete series of $SU(1,2)$ then one has
 \begin{itemize}
  \item $\{\xi\in\mathfrak n^*, \xi(X_{2\alpha})\geq 0 \}\subset \ac(\cO-\supp \pi_{|N}) $ if $\pi$ is a holomorphic discrete series
  \item $\{\xi\in\mathfrak n^*, \xi(X_{2\alpha})\leq 0 \}\subset\ac(\cO-\supp \pi_{|N}) $ if $\pi$ is a anti-holomorphic discrete series
  \item $\ac(\cO-\supp \pi_{|N}) = \mathfrak n^*$ if $\pi$ is a non-holomorphic discrete series
 \end{itemize}
\end{proposition}
Note that with completely different methods the restriction of discrete series of $SU(2,1)$ to maximal solvable groups $AN$ has been studied in \cite{Liu12}.
\begin{proof}
The Rossmann orbit of a discrete series representation is precisely the coadjoint orbit $\mathrm{Ad}^*(G)\Lambda\subset \mathfrak g^*$ through the Harish-Chandra parameter $\Lambda \in \mathfrak h^*$ \cite{Ros78}, i.e. by \cite[Theorem 1.2]{harrisheolaf} we know $\mathrm{WF}(\pi) = \mathrm{AC}(\mathcal \mathrm{Ad}^*(G)\Lambda)$. 
Furthermore, in complete generality, we know that for a closed subgroup $H\subset G$ one has $\mathrm{WF}(\pi_{|H}) \supset \mathrm{pr}_{\mathfrak g^*\to\mathfrak h^*} WF(\pi)$ \cite[Proposition 1.5]{howe}. 
One is thus left with the exercise of calculating $\mathrm{pr}_{\mathfrak g\to\mathfrak n} (\mathrm{Ad}^*(G)\Lambda)$ or (as the coadjoint orbits in $N$ are determined by the evaluation at $X_{2\alpha}$ that spans the center of $N$) it is enough to calculate the range of $G\ni g\mapsto \langle \mathrm{Ad}^*(g) \Lambda,X_{2\alpha}\rangle \subset \mathbb R$. 
The latter can easily be achieved by considering the Iwasawa decomposition $G=NAK$. 
As $X_{2\alpha}$ is in the center of $N$ and by definition of the root spaces we have 
$\langle\mathrm{Ad}^*(nak)\Lambda,X_{2\alpha}\rangle =e^{2\alpha\log a} \langle\mathrm{Ad}^*(k)\Lambda,X_{2\alpha}\rangle$, thus it remains to calculate the sign of $\langle\mathrm{Ad}^*(k)\Lambda,X_{2\alpha}\rangle$.
Using the $\theta$ invariance of $\mathfrak k$ and $\langle\bullet,\bullet\rangle$ and the definition of $\alpha_2$ we get
\[
\langle\mathrm{Ad}^*(k)\Lambda,X_{2\alpha}\rangle = \langle\mathrm{Ad}^*(k)\Lambda,\frac12(X_{2\alpha}+\theta X_{2\alpha})\rangle=\langle\mathrm{Ad}^*(k)\Lambda,T_2\rangle
\]
Now note that the Weyl group of $K$ precisely swaps $\alpha_1$ and $\alpha_2$ and it follows from Kostants convexity theorem \cite{Kos73} that the projection of a coadjoint orbit in a compact Lie algebra to a Cartan subgroup is a Weyl group invariant polytope.
Thus $\langle\mathrm{Ad}^*(k)\Lambda,X_{2\alpha}\rangle\geq 0$ iff $\langle\Lambda, \alpha_1\rangle\geq 0$ and $\langle\Lambda, \alpha_2\rangle \geq 0 $ and
$\langle\mathrm{Ad}^*(k)\Lambda,X_{2\alpha}\rangle\leq 0$ iff $\langle\Lambda, \alpha_1\rangle\leq 0$ and $\langle\Lambda, \alpha_2\rangle\leq 0$.
This shows, that
\begin{itemize}
 \item $\overline{\mathrm{pr}_{\mathfrak g\to\mathfrak n} (\mathrm{Ad}^*(G)\Lambda)} = \{\xi\in\mathfrak n^*, \xi(X_{2\alpha})\geq 0 \}$ if $\pi$ is a holomorphic discrete series
 \item $\overline{\mathrm{pr}_{\mathfrak g\to\mathfrak n} (\mathrm{Ad}^*(G)\Lambda)} = \{\xi\in\mathfrak n^*, \xi(X_{2\alpha})\leq 0 \}$ if $\pi$ is a anti-holomorphic discrete
 series
 \item $\overline{\mathrm{pr}_{\mathfrak g\to\mathfrak n} (\mathrm{Ad}^*(G)\Lambda)} = \mathfrak n$ if $\pi$ is a non-holomorphic discrete series.
\end{itemize}

\end{proof}
\providecommand{\bysame}{\leavevmode\hbox to3em{\hrulefill}\thinspace}
\providecommand{\MR}{\relax\ifhmode\unskip\space\fi MR }
\providecommand{\MRhref}[2]{%
  \href{http://www.ams.org/mathscinet-getitem?mr=#1}{#2}
}
\providecommand{\href}[2]{#2}

\end{document}